\documentclass[reqno,twoside,12pt]{amsart}
\usepackage{amssymb,amsfonts,amsthm,amsmath,mathrsfs}
\usepackage{enumitem}

\usepackage[pagebackref=false]{hyperref}

\usepackage[hmargin=1in,vmargin=1in]{geometry}

\usepackage{cite}
\usepackage{color}

\def\eqdef{\stackrel{\rm def}{=}}

\def\beq{\begin{equation}}
\def\eeq{\end{equation}}
\def\beqs{\begin{equation*}}
\def\eeqs{\end{equation*}}

\newcommand{\tnum}{\rm(\roman*)}
\newcommand{\rnum}{\rm(\alph*)}

\newtheorem{theorem}{Theorem}[section]
\newtheorem{lemma}[theorem]{Lemma}
\newtheorem{proposition}[theorem]{Proposition}
\newtheorem{corollary}[theorem]{Corollary}
\newtheorem{definition}[theorem]{Definition}
\newtheorem{assumption}[theorem]{Assumption}

\theoremstyle{definition}
\newtheorem{remark}[theorem]{Remark}
\newtheorem{example}[theorem]{Example}

\newcommand{\ddt}{\frac{\d}{\d t}}
\def\d{{\rm d}}

\def\varep{\varepsilon}

\renewcommand{\Re}{\operatorname{Re}}
\renewcommand{\Im}{\operatorname{Im}}

\renewcommand{\geq}{\ge}

\newcommand{\R}{\ensuremath{\mathbb R}}
\newcommand{\C}{\ensuremath{\mathbb C}}
\newcommand{\N}{\ensuremath{\mathbb N}}
\newcommand{\Z}{\ensuremath{\mathbb Z}}

\newcommand{\K}{\mathbb{K}}

\newcommand{\LL}{\mathcal L}
\newcommand{\iln}{L}

\newcommand{\bigo}{\mathcal O}

\numberwithin{equation}{section}

\title{A new form of asymptotic expansion for non-smooth differential equations with time-decaying forcing functions}
\author{Luan Hoang}

\address{Department of Mathematics and Statistics,
Texas Tech University\\
1108 Memorial Circle, Lubbock, TX 79409--1042, U. S. A.}
\email{luan.hoang@ttu.edu}

\makeatletter
\@namedef{subjclassname@2020}{
  \textup{2020} Mathematics Subject Classification}
\makeatother


\date{\today}

\begin{document}

\begin{abstract}
This article is focused on the asymptotic expansions, as time tends to infinity, of solutions of a  system of ordinary differential equations with non-smooth nonlinear terms. The forcing function decays to zero in a very complicated but coherent way. We prove that every decaying solution admits an asymptotic expansion of a new type. This expansion contains a new variable that allows it to be established in a closed-form, but does not affect the meaning and precision of the expansion. Moreover, the expansion is  constructed explicitly with the use of the complexification method.\end{abstract}

\maketitle


\pagestyle{myheadings}\markboth{L. Hoang}
{A New Form of Asymptotic Expansion for Non-smooth Differential Equations}

\tableofcontents

\section{Introduction}\label{intro}

Describing the long-time dynamic of solutions of nonlinear ordinary differential equations (ODE) and partial differential equations (PDE) is always difficult. Even when its general knowledge is available such as the existence of the global attractor or stability of an equilibrium, understanding it in finer details is challenging. However, there have been progresses over the years for the latter case. Below, we review some of them as the motivation for our current research.
The papers  \cite{FS84a,FS87} study the Navier--Stokes equations (NSE) for a viscous, incompressible fluid in a bounded or periodic domain $\Omega$ in $\R^3$  which can be written in the functional form as
\beq\label{NSE}
\frac{\d y}{\d t}+Ay+B(y,y)=f(t),
\eeq
where $A$ is the (linear) Stokes operator, $B$ is a bilinear form in appropriate functional spaces, and $f(t)$ is the Leray projection of the outer body force.

When the body force is potential, i.e., $f(t)\equiv 0$,
Foias and Saut prove in \cite{FS84a} that any non-trivial regular solution $y(t)$ of \eqref{NSE} satisfies, as $t\to\infty$,  
\beq\label{rawxi}
e^{\lambda t} y(t)\to \xi\text{ for some $\lambda>0$ and $\xi\ne 0$ with } A\xi=\lambda\xi.
\eeq
Furthermore, they obtain in \cite{FS87} the following  asymptotic expansion, as $t\to\infty$,
\beq\label{FSx}
y(t)\sim \sum_{n=1}^\infty q_n(t) e^{-\mu_n t},
\eeq
in $C^m(\bar \Omega)^3$ for any integer $m\ge 0$, where $q_n(t)$ are polynomials in $t$, valued in the space of smooth functions, the sequence $(\mu_n)_{n=1}^\infty$ is positive, strictly increasing to infinity.
More precisely, for any integers $N\ge 1$ and  $m\ge 0$, there is a number $\mu>\mu_N$ such that 
$$ \left\| y(t)- \sum_{n=1}^N q_n(t) e^{-\mu_n t}\right\|_{C^m(\bar \Omega)^3}=\bigo(e^{-\mu t})\text{ as }t\to\infty.$$
The asymptotic expansion \eqref{FSx} leads to further developments such as the associated nonlinear spectral manifold and Poincar\'e--Dulac normal form among  other things, see \cite{FHS2} and references therein.
The asymptotic expansion of the type \eqref{FSx} is also established for the NSE of rotating fluids  \cite{HTi1},  dissipative wave equations \cite{Shi2000}, and the Lagrangian trajectories  \cite{H4}.

When the body force is non-potential, under the assumption that $f(t)\to 0$ as $t\to\infty$ in a coherent manner
\beq \label{fxp0}
    f(t)\sim\sum_{k=1}^\infty f_k(t),
\eeq
the solution $y(t)$ of \eqref{NSE} is proved to admit a similar asymptotic expansion
\beq \label{yxp0}
    y(t)\sim \sum_{k=0}^\infty y_k(t),
\eeq
see \cite{HM2,CaH1,CaH2,H6}. The meaning of the expansions \eqref{fxp0} and \eqref{yxp0} above is similar to that of \eqref{FSx}, but their forms can be much more general. The reader is referred to the cited papers for details.
The results \eqref{rawxi}, \eqref{FSx} and \eqref{yxp0} are extended 
to general systems of ODE without forcing functions in \cite{Minea,CaHK1,H7,H8,H9}, and with forcing functions in \cite{CaH3,H5}. 
(For classical theory of and other approaches to the asymptotic analysis of ODE, see \cite{ArnoldODEGeo,BrunoBook1989,BrunoBook2000,LefschetzBook,KFbook2013,WasowBook}.)
We recall here some  results and ideas from \cite{H5,CaHK1} that are relevant to the current work.

The paper \cite{H5} studies the following ODE system in $\R^n$ or $\C^n$, which is in a more general form than \eqref{NSE},
\beq\label{mainode}
\frac{\d y}{\d t} +Ay =F(y)+f(t),
\eeq
where $A$ is an $n\times n$ constant matrix, and $F$ is a vector field on $\R^n$ satisfying $F(0)=0$, $F'(0)=0$, and having a the Taylor expansion about the origin.
The function $f(t)$ is assumed to have  a very general expansion \eqref{fxp0}, where 
\beq \label{fkreview}
\text{    $f_k$ may contain complex number powers of $e^t$, $t$, $\ln t$, $\ln\ln t$, etc.}
\eeq
Then any decaying solution $y(t)$ admits the expansion \eqref{yxp0}.

The paper \cite{CaHK1} studies \eqref{mainode} in $\R^n$ under the considerations that $f(t)\equiv 0$
and the function $F(y)$ may not be smooth in any neighborhood of the origin. However, $F$ is assumed to have the following asymptotic expansion
\beq \label{Fsum}
F(y)\sim \sum_{k=1}^\infty F_k(y)\text{ as }y\to 0,
\eeq 
where 
\beq \label{NSH}
    F_k\in C^\infty(\R^n\setminus\{0\})\text{   is positively homogeneous of a degree  $\beta_k>1$ with $\beta_k\nearrow \infty$.}
\eeq 
After establishing the first asymptotic approximation \eqref{rawxi}, one can scale the solution and shift the center of expansion of the Taylor series of $F_k(y)$ from the origin to $\xi\ne 0$. 
Then  the techniques in \cite{FS87,CaH3} are utilized to obtain the asymptotic expansion \eqref{FSx}.

\medskip
The current paper aims to study equation \eqref{mainode} but with the hypotheses \eqref{fxp0}, \eqref{fkreview}, \eqref{Fsum} and  \eqref{NSH} which are more general than all the cited papers. Our goal is to obtain the  asymptotic expansion \eqref{yxp0}, as $t\to\infty$,  for any decaying solution $y(t)$. As pointed out in section \ref{newex} below, in some circumstances, a direct application of the above reviewed methods cannot achieve this goal. To fix it, we have to introduce a new variable $\zeta=\zeta(t)$, which will be called subordinate variable, and prove that any decaying solution $y(t)$ has an asymptotic expansion
\beq\label{newye}
y(t)\sim \sum_{k=1}^\infty Y_k(t,\zeta(t)), \text { as } t\to\infty,
\eeq 
where each function $Y_k(t,\zeta)$ can be constructed recursively.
Of course, the meaning of the new asymptotic expansion \eqref{newye} needs to be justified.
With the new variable $\zeta$, we have to resolve many technical issues in order to achieve \eqref{newye}.

\medskip
The paper is organized as follows.
In section \ref{bkgsec}, we recall previously studied classes of functions and asymptotic expansions. A new form of asymptotic expansion is defined in Definition \ref{zetxp} using a new variable, called the subordinate variable, $\zeta$.
We state the main result -- Theorem \ref{mainthm} -- in section \ref{mainresult}.
Section \ref{tools} establishes fundamental estimates that are needed in later proofs.
Particularly, we have  the fundamental asymptotic approximation result for linear equations in Theorem \ref{newlin}. 
The asymptotic behavior of the solutions of equation \eqref{mainode} are obtained in section \ref{asympsec}.
More specifically, the first asymptotic approximation, even for general equations in $\C^n$, is established in Theorem \ref{thmap1}.
The asymptotic expansion is proved in Theorem \ref{mainthm2}  which is a more detailed version of the Main 
Theorem \ref{mainthm}. 
In dealing with systems of differential equations with variables and solutions being real numbers, it is intuitive to have the results stated with only real number variables and $\R$-valued functions. Therefore, Theorem \ref{thmreal} in section \ref{sinsec} is obtained purely using just those, particularly, the real sinusoidal functions. 
Conclusions and remarks on future developments are made in section \ref{conclude}.
\section{Preliminaries}\label{bkgsec}

\subsection{Notation}
We  use the following notation throughout the paper. 
\begin{itemize}
 \item $\N=\{1,2,3,\ldots\}$, $\Z_+=\N\cup\{0\}$ and $i=\sqrt{-1}$.

 \item For a vector $x\in\C^n$, we denote by $|x|$, $\Re x$ and $\Im x$  its Euclidean norm, real and imaginary parts, and by $x^{(k)}$ the $k$-tuple $(x,\ldots,x)$ for $k\in\N$, and  $x^{(0)}=1$. 
 For an $m\times n$ matrix $M$ of complex numbers, its Euclidean norm in $\C^{mn}$ is denoted by $|M|$.
\item We will use the convention 
$\sum_{k=1}^0 a_k=0$ and $\prod_{k=1}^0 a_k=1$.
 
 \item Let $f$ be a $\C^m$-valued function and $h$ be a non-negative function, both  are defined in a neighborhood of the origin in $\C^n$. 
 We write 
 $f(x)=\bigo(h(x))$ as $x\to 0$
to indicate $|f(x)|=\bigo(h(x))\text{  as }x\to 0.$
 
 \item Let $f,g:[T_0,\infty)\to \C^n$ and $h:[T_0,\infty)\to[0,\infty)$ for some $T_0\in \R$. We write 
 $$f(t)=\bigo(h(t)), \text{ implicitly meaning as $t\to\infty$,} $$
if $|f(t)|=\bigo(h(t))$ as $t\to\infty$.
Also,
$$f(t)=g(t)+\bigo(h(t))\text{ means } f(t)-g(t)=\bigo(h(t)).$$  

\item If $n,m,k\in \N$ and $\mathcal L$ is an $m$-linear mapping from $(\R^n)^m$ to $\R^k$, the norm of $\mathcal L$ is defined by
\beqs
\|\mathcal L\|=\max\{ |\mathcal L(x_1,x_2,\ldots,x_m)|:x_j\in\R^n,|x_j|=1,\text{ for } 1\le j\le m\}.
\eeqs 
It is known that the norm $\|\mathcal L\|$ belongs to  $[0,\infty)$, and one has 
\beq\label{multiL}
|\mathcal L(x_1,x_2,\ldots,x_m)|\le \|\mathcal L\|\cdot |x_1|\cdot |x_2|\ldots |x_m|
\quad \forall x_1,x_2,\ldots,x_m\in\R^n.
\eeq

\item
Let $S$ be a subset of $\C^k$ with $k\in\N$. We say $S$ preserves the conjugation, respectively, the addition, if the conjugate $\bar x$ of any $x\in S$ also belongs to $S$,  respectively, if $x+y\in S$ for all $x,y\in S$.
 When $k=1$, we say $S$ preserves the unit increment if $x+1\in S$ for all $x\in S$.

\end{itemize}

\subsection{Elementary  functions}
In this paper, we only deal with single-valued complex functions. Particularly, for $z\in\C$ and $t>0$, we have the  power function
$t^z=\exp(z\ln t)=e^{z\ln t}$.
We will often use the fact $|t^z|=t^{\Re z}$.

\begin{definition}\label{ELdef} Define the iterated exponential and logarithmic functions as follows:
\begin{align*} 
&E_0(t)=t \text{ for } t\in\R,\text{ and } E_{m+1}(t)=e^{E_m(t)}  \text{ for } m\in \Z_+, \ t\in \R,\\
&\iln_{-1}(t)=e^t,\quad \iln_0(t)=t\text{ for } t\in\R,\text{ and }\\
& \iln_{m+1}(t)= \ln(\iln_m(t)) \text{ for } m\in \Z_+,\ t>E_m(0).
\end{align*}

For $k\in \Z_+$, define $\widehat \LL_k=(\iln_{-1},\iln_{0},\iln_{1},\ldots,\iln_{k})$, that is,
\beqs  
\widehat \LL_k(t)=(e^t,t,\ln t,\ln\ln t,\ldots,\iln_{k}(t)).
\eeqs 
\end{definition}

For $m\in\Z_+$, note that $\iln_m(t)$ is positive and increasing  for $t>E_m(0)$ and  
$\iln_m(E_{m+1}(0))=1$.
Clearly,\beq  \label{LLk}
\lim_{t\to\infty} \frac{\iln_k(t)^\lambda}{\iln_m(t)}=0\text{ for all }k>m\ge -1 \text{ and } \lambda\in\R.
\eeq 

We recall a useful inequality from \cite{CaH3}.

\begin{lemma}[{\cite[Lemma 2.5]{CaH3}}]\label{plnlem}
Let $m\in\Z_+$ and $\lambda>0$, $\gamma>0$  be given. For any number $T_*>E_m(0)$, there exists a number $C>0$ such that
\beq\label{iine2}
 \int_0^t e^{-\gamma (t-\tau)}\iln_m(T_*+\tau)^{-\lambda}\d\tau
 \le C \iln_m(T_*+t)^{-\lambda} \quad\text{for all }t\ge 0.
\eeq
\end{lemma}

\subsection{Previous asymptotic expansions}\label{oldex}
We recall the definitions in \cite[Subsection 3.2]{H6}.
Let $\K=\R$ or $\C$. For 
\beq
\label{azvec} 
z=(z_{-1},z_0,z_1,\ldots,z_k)\in (0,\infty)^{k+2}
\text{ and } 
\alpha=(\alpha_{-1},\alpha_0,\alpha_1,\ldots,\alpha_k)\in \K^{k+2},
\eeq
 define 
$\begin{displaystyle} z^{\alpha}=\prod_{j=-1}^k z_j^{\alpha_j}.
 \end{displaystyle}$
The vector $\alpha$ is called the power vector. We say $\alpha$ is real if $\alpha=\Re(\alpha)$.

For $\mu\in\R$, $m,k\in\Z$ with  $k\ge m\ge -1$, denote by $\mathcal E_\K(m,k,\mu)$ the set of vectors $\alpha$ in \eqref{azvec} 
 such that
 \beqs
\Re(\alpha_j)=0 \text{ for $-1\le j<m$   and  } \Re(\alpha_m)=\mu.
\eeqs 
Particularly, $\mathcal E_\R(m,k,\mu)$ is the set of vectors $\alpha=(\alpha_{-1},\alpha_0,\ldots,\alpha_k)\in \R^{k+2}$   such that
 $$\alpha_{-1}=\ldots=\alpha_{m-1}=0 \text{ and }\alpha_m=\mu.$$
For example, when $m=-1$, $k\ge -1$, $\mu=0$, the set 
\beq \label{Eminus}
\mathcal E_\K(-1,k,0)\text{  is  the collection of vectors $\alpha$ in \eqref{azvec} with $\Re (\alpha_{-1})=0$. }
\eeq 

Let $k\ge m\ge -1$, $\mu\in\R$, and $\alpha\in\mathcal E_\K(m,k,\mu)$.
Using \eqref{LLk}, one can verify that, see, e.g., equation (3.14) in \cite{H5}, 
\beq\label{LLo}
\lim_{t\to\infty} \frac{\widehat\LL_{k}(t)^\alpha}{\iln_{m}(t)^{\mu+\delta}}=0 \quad\text{ for any }\delta>0.
\eeq

 \begin{definition}\label{Fclass}
 Let $\K$ be $\R$ or $\C$,  and  $X$ be a linear space over $\K$.

\begin{enumerate}[label=\tnum]
\item For $k\ge -1$, define $\mathscr P(k,X)$ to be the set of functions of the form 
\beq\label{pzdef} 
p(z)=\sum_{\alpha\in S}  z^{\alpha}\xi_{\alpha}\text{ for }z\in (0,\infty)^{k+2},
\eeq 
where $S$ is some finite subset of $\K^{k+2}$, and each $\xi_{\alpha}$ belongs to $X$.

\item Let $k\ge m\ge -1$ and $\mu\in\R$. 
Define $\mathscr P_{m}(k,\mu,X)$ to be the set of functions of the form \eqref{pzdef},
where $S$ is a finite subset of $\mathcal E_\K(m,k,\mu)$ and each $\xi_{\alpha}$ belongs to $X$.
\end{enumerate}
 \end{definition}

If $(X,\|\cdot\|_X)$ is a normed space over $\K=\R$ or $\C$ and  $p\in \mathscr P_{m}(k,\mu,X)$, then, as a consequence of \eqref{LLo}, one has
\beq\label{LLp}
\lim_{t\to\infty} \frac{\|p(\widehat\LL_{k}(t))\|_X}{\iln_{m}(t)^{\mu+\delta}}=0 \quad\text{ for any }\delta>0.
\eeq

\begin{definition}\label{Lexpand}
Let $\K$ be $\R$ or $\C$, and $(X,\|\cdot\|_X)$ be a normed space over $\K$.
Suppose $g(t)$ is a function from $(T,\infty)$ to $X$ for some $T\in\R$, and $m_*\in \Z_+$. 

\begin{enumerate}[label={\tnum}]
 \item\label{LE1} Let $(\gamma_k)_{k=1}^\infty$ be a divergent, strictly increasing sequence of positive numbers, and $(n_k)_{k=1}^\infty$ be a sequence in $\N\cap[m_*,\infty)$. 
We say
\beq\label{expan1}
g(t)\sim \sum_{k=1}^\infty p_k(\widehat{\LL}_{n_k}(t)), \text{ where $p_k\in \mathscr P_{m_*}(n_k,-\gamma_k,X)$ for $k\in\N$, }
\eeq
if, for each $N\in\N$, there is some number $\mu>\gamma_N$ such that
\beqs
\left\|g(t) - \sum_{k=1}^N p_k(\widehat{\LL}_{n_k}(t))\right\|_X=\bigo(\iln_{m_*}(t)^{-\mu}).
\eeqs

\item\label{LE2} Let $N\in\N$, $(\gamma_k)_{k=1}^N$ be positive and strictly increasing, and  $n_*\in\N\cap[m_*,\infty)$.
We say
\beqs
g(t)\sim \sum_{k=1}^N p_k(\widehat{\LL}_{n_*}(t)), \text{ where $p_k\in \mathscr P_{m_*}(n_*,-\gamma_k,X)$ for $1\le k\le N$, }
\eeqs
if it holds for all $\mu>\gamma_N$ that
\beq\label{gremain2}
\left\|g(t) - \sum_{k=1}^N p_k(\widehat{\LL}_{n_*}(t))\right\|_X=\bigo(\iln_{m_*}(t)^{-\mu}).
\eeq
\end{enumerate}
\end{definition}
 
\subsection{A new type of asymptotic expansions} \label{newex}
To motivate the new concept, we consider a simple example for equation \eqref{mainode} when $A$ is a real diagonal matrix, 
$$F(y)=|y|^{3/2}Z_0\text{ and }
f(t)=( (\ln t)^{1/2} - (\ln\ln t)^{-3} +1)t^{-1} \xi_0,\quad 
Z_0,\xi_0\in \R^n\setminus\{0\}.$$
Denote $\zeta(t)=(\ln t)^{1/2} - (\ln\ln t)^{-3} +1$. On the one hand, by using the techniques in \cite{CaH3,H5}, we can obtain the first asymptotic approximation, as $t\to\infty$,
\beq \label{guess1} 
y(t)= \zeta(t)t^{-1}\xi_*+R(t),\text{ where $\xi_*\ne 0$, $R(t)=\bigo(t^{-1-\delta})$ for some $\delta>0$.}
\eeq
On the other hand, following \cite{CaHK1,H5}, we may expect to have the asymptotic expansion \eqref{yxp0} where each $y_k(t)$ is a \textit{finite} sum of the functions
\beq \label{guess2}
    t\mapsto (\ln t)^\alpha (\ln\ln t)^\beta t^{-\mu_k}\xi,\text{ where }
\alpha,\beta\in \R,\mu_k>0,\xi\in\R^n.
\eeq 
Combining \eqref{guess1} and \eqref{guess2}, we attempt
\beq \label{guess3}
    y(t)=\zeta(t)t^{-1}\xi_*+y_2(t)+\ldots
\eeq
Same as in \cite{CaHK1}, we do the scaling and shifting next.
Letting $\theta(t)=\zeta(t)t^{-1}$, we rewrite $y(t)=\theta(t)\left ( \xi_*+\theta(t)^{-1}y_2(t)+\ldots\right)$ and 
\begin{align*}
    F(y(t))&=\theta(t)^{3/2}|\xi_*+\theta(t)^{-1} y_2(t)+\ldots|^{3/2}Z_0\\
&=\theta(t)^{3/2}\left\{ |\xi_*|^{3/2}+ \theta(t)^{-1} y_2(t)\cdot\frac{3\xi_*}{2|\xi_*|^{1/2}}  +\ldots\right\}Z_0 \\
&=\left\{\theta(t)^{3/2} |\xi_*|^{3/2}+ \theta(t)^{1/2} y_2(t)\cdot\frac{3\xi_*}{2|\xi_*|^{1/2}}  +\ldots\right\}Z_0.
\end{align*} 
Observe that  the terms $\theta(t)^{3/2}$ and $\theta(t)^{1/2}$ cannot be expressed as  finite sums of the functions in \eqref{guess2}. Therefore, using the above form \eqref{guess3} with $y_2(t)$ being a finite sum of the functions in \eqref{guess2}  to match both sides of \eqref{mainode} fails.

The new idea to fix this issue is to replace the function in \eqref{guess2} with 
\beq \label{guess4}
    t\mapsto \zeta(t)^{\gamma}(\ln t)^\alpha (\ln\ln t)^\beta t^{-\mu_k}\xi\text{ for }
\alpha,\beta,\gamma\in \R.
\eeq 
We emphasize the presence of the new term $\zeta(t)$ in \eqref{guess4}.
We need to verify that 
\begin{enumerate}[label=\tnum]
    \item all $y_k$ in \eqref{yxp0} can actually be obtained as a finite sum of the functions in \eqref{guess4}, and
    \item the asymptotic expansion \eqref{yxp0} is still meaningful.
\end{enumerate}
In general cases, these two goals will, indeed, be achieved under suitable conditions on the asymptotic expansion \eqref{fxp0} of $f(t)$. To specify these conditions, we introduce a new class of functions that will accommodate a new variable $\zeta$.
 \begin{definition}\label{Fsubo}
 Let $\K$ be $\C$ or $\R$,  and  $X$ be a linear space over $\K$.

\begin{enumerate}[label=\tnum]
\item For $k\ge -1$, define $\widehat{\mathscr P}(k,X)$ to be the set of functions of the form 
\beq\label{pzedef} 
p(z,\zeta)=\sum_{(\alpha,\beta)\in S}  z^{\alpha}\zeta^\beta\xi_{\alpha,\beta}\text{ for }z\in (0,\infty)^{k+2},\zeta\in(0,\infty),
\eeq 
where $S$ is some finite subset of $\K^{k+2}\times\R$, and each $\xi_{\alpha,\beta}$ belongs to $X$.

\item Let $k\ge m\ge -1$ and $\mu\in\R$. 
Define $\widehat{\mathscr P}_{m}(k,\mu,X)$ to be set of functions of the form \eqref{pzedef},
where $S$ is a finite subset of ${\mathcal E}_\K(m,k,\mu)\times\R$ and each $\xi_{\alpha,\beta}$ belongs to $X$.
 \end{enumerate}
 \end{definition}

We call $\zeta$ in Definition \ref{Fsubo} the \textit{subordinate variable}.
Below are immediate observations about Definitions \ref{Fclass} and \ref{Fsubo}. 
Let 
$$(\tilde{\mathscr P}=\mathscr P,\hat k=k,\hat z=z)\text{ or }
(\tilde{\mathscr P}=\widehat{\mathscr P},\hat k=k+1,\hat z=(z,\zeta)).$$ 
\begin{enumerate}[label=(\alph*)]
 \item\label{Cb} Each $\tilde{\mathscr P}(k,X)$ is a linear space over $\K$.

\item Clearly,
$${\mathscr P}(k,X)\subset \widehat{\mathscr P}(k,X)\text{ and } {\mathscr P}_{m}(k,\mu,X)\subset \widehat{\mathscr P}_{m}(k,\mu,X).$$

 \item\label{Ca} $\tilde{\mathscr P}(k,X)$ contains all polynomials from $\R^{\hat k+2}$ to $X$, in the sense that, if $p:\R^{\hat k+2}\to X$ is a polynomial, then its restriction on $(0,\infty)^{\hat k+2}$ belongs to $\tilde{\mathscr P}(k,X)$.

 \item\label{Cc} If $k'>k\ge -1$, then, by the standard embedding
 $$ \K^{k+2}=\K^{k+2}\times\{0\}^{k'-k}\subset \K^{k'+2},$$ 
 one can embed  $\tilde{\mathscr P}(k,X)$ into $\tilde{\mathscr P}(k',X)$. 
 
 \item\label{Cd} One has
 \beq\label{qpequiv}
 \begin{aligned}
& q\in  \tilde{\mathscr P}_{m}(k,\mu,X) \text{ if and only if } \exists p\in  \tilde{\mathscr P}_{m}(k,0,X), q(\hat z)\equiv p(\hat z) z_m^{\mu}.
 \end{aligned}
 \eeq
In \eqref{qpequiv} above, $z=(z_{-1},z_0,\ldots,z_k)\in (0,\infty)^{k+2}$.

 \item \label{Ce} For any $k\ge m\ge 0$ and $\mu\in\R$, one has
\beq\label{Pm10}
\tilde{\mathscr P}_m(k,\mu,X)\subset \tilde{\mathscr P}_{-1}(k,0,X) . 
\eeq
\end{enumerate}

We recall Definition 10.6 of \cite{H5} which makes use of a real inner product space $X$ and its complexification $X_\C$. For the latter, see \cite[Section 77]{Halmos1974} or the review in \cite[Section 10]{H5}. For the purpose of this paper, we can simply consider $X=\R^n$ and $X_\C=\C^n$ for some $n\ge 1$.

 \begin{definition}\label{realF}
Let $X$ be an inner product space over $\R$, and $X_\C$ be its complexification.

\begin{enumerate} [label=\tnum]
\item For $k\ge -1$, define $\widehat{\mathscr P}(k,X_\C,X)$ to be the set of functions of the form  \eqref{pzedef},
where $S$ is a finite subset of $\C^{k+2}\times \R$ that preserves the conjugation,
and each $\xi_{\alpha,\beta}$ belongs to $X_\C$, with the conjugation condition
\beq\label{xiconj}
\xi_{\bar \alpha,\beta}=\overline{\xi_{\alpha,\beta}} \quad\forall (\alpha,\beta)\in S.
\eeq

\item For $k\ge m\ge -1$ and $\mu\in\R$, define $\widehat{\mathscr P}_{m}(k,\mu,X_\C,X)$ to be the set of functions in $\widehat{\mathscr P}(k,X_\C,X)$ with the restriction that $\alpha \in \mathcal E_\C(m,k,\mu)$ for any $(\alpha,\beta)\in S$.

\item The classes ${\mathscr P}(k,X_\C,X)$ and ${\mathscr P}_{m}(k,\mu,X_\C,X)$  are defined as in {\rm (i)} and {\rm (ii)}, respectively, but without $\zeta$ and $\beta$, that is, $(z,\zeta)$, $(\alpha,\beta)$, $\xi_{\alpha,\beta}$, $\C^{k+2}\times \R$
are replaced with $z$, $\alpha$, $\xi_\alpha$, $\C^{k+2}$, and $p=p(z)$. 
 \end{enumerate}
 \end{definition}

Because $S$ preserves the conjugation, if $(\alpha,\beta)\in S$, then
$(\bar\alpha,\beta)=\overline{(\alpha,\beta)}\in S$, hence, the term $\xi_{\bar\alpha,\beta}$ in \eqref{xiconj} exists.
Thanks additionally to the conjugation condition \eqref{xiconj}, we can rewrite $p(z,\zeta)$ in \eqref{pzedef} as
\beq\label{phalf} 
p(z,\zeta)
=\sum_{(\alpha,\beta)\in S} \frac12 \left ( z^{\alpha}\zeta^\beta\xi_{\alpha,\beta}
+  z^{\bar \alpha}\zeta^{\bar \beta}\xi_{\bar \alpha,\bar \beta} \right )
=\sum_{(\alpha,\beta)\in S} \frac12 \left ( z^{\alpha}\zeta^\beta\xi_{\alpha,\beta}
+ \overline{ z^{\alpha}\zeta^\beta\xi_{\alpha,\beta}} \right )
\eeq
for $z\in (0,\infty)^{k+2},\zeta\in(0,\infty)$.
Hence, each function $p$ in the class $\widehat{\mathscr P}(k,X_\C,X)$ is, in fact, $X$-valued.

It is clear that 
$${\mathscr P}(k,X_\C,X)\subset \widehat{\mathscr P}(k,X_\C,X)\text{ and }
{\mathscr P}_{m}(k,\mu,X_\C,X)\subset \widehat{\mathscr P}_{m}(k,\mu,X_\C,X).$$
Let $\tilde{\mathscr P}=\mathscr P$ or $\widehat{\mathscr P}$. Then
$$\tilde{\mathscr P}(k,X_\C,X)\subset \tilde{\mathscr P}(k,X_\C)\text{ and }
\tilde{\mathscr P}_{m}(k,\mu,X_\C,X)\subset \tilde{\mathscr P}_{m}(k,\mu,X_\C).$$
We also observe that the classes $\tilde{\mathscr P}(k,X_\C,X)$  and  $\tilde{\mathscr P}_{m}(k,\mu,X_\C,X)$ are (additive) subgroups of $\tilde{\mathscr P}(k,X_\C)$, but not linear spaces over $\C$.

In order for the scaling and shifting techniques in \cite{CaHK1} to work for our current problem, the leading term in the asymptotic expansion \eqref{fxp0} of $f(t)$ needs to be in a certain class of functions which we describe below. We will use the lexicography order for the real parts  of the power vectors $\alpha$ in \eqref{pzedef}. 

\begin{definition}\label{Pplus}
Define $\mathscr P^{+}(k)$ to be a subset of $\mathscr P(k,\C,\R)$ such that each element $p \in \mathscr P^{+}(k)$ has the form 
\beq \label{pzmax}
p(z)=z^{\alpha_{\max}} c_{\alpha_{\max}}+q(z),
\eeq 
where the power vector $\alpha_{\max}$ is real, the number $c_{\alpha_{\max}}$ is positive, the function $q$ belongs to  $\mathscr P(k,\C,\R)$ and has the representation as in \eqref{pzdef} with all $\Re(\alpha)<\alpha_{\max}$.

For $k\ge m\ge -1$, define 
\beq \label{strictPp}
    \mathscr P_m^{+}(k)=\{p\in \mathscr P^{+}(k)\cap \mathscr P_m(k,0,\C,\R):\text{ all power vector $\alpha$ in \eqref{pzdef} has $\alpha_{-1}=0$}\}.
\eeq
\end{definition}

The following remarks are in order.
\begin{enumerate}[label=\rnum]
    \item Let $p\in {\mathscr P}^+(k)$. Then there exist numbers $T>E_k(0)$ and $c\ge 1$ such that
\beq\label{plusest}
c^{-1} \widehat{\LL}_k(t)^{\alpha_{\max}}\le p(\widehat{\LL}_k(t))\le c \widehat{\LL}_k(t)^{\alpha_{\max}} \text{ for all } t\ge T,
\eeq 
and, consequently,
\beq\label{pluspos}
p(\widehat{\LL}_k(t))>0 \text{ for all } t\ge T.
\eeq 

\item Let $k\ge m\ge -1$ and $p\in {\mathscr P}_m^+(k)$. 
Note that the power vector $\alpha_{\max}$ in \eqref{pzmax} belongs to ${\mathcal E}_\R(m,k,0)$.
Thanks to this fact, \eqref{LLo} and \eqref{plusest}, one has
\beq \label{plusrate}
    p(\widehat{\LL}_k(t))^s=\bigo(\iln_m(t)^\delta) \text{ for any numbers  $s\in\R$ and $\delta>0$.}
\eeq

    \item Suppose $q(z)=p(z)\xi$, where $\xi$ is a vector in some linear space $X$, and the function $p$ satisfies the same conditions as in the definition of ${\mathscr P}^+(k)$ except that the coefficient $c_{\alpha_{\max}}$ in \eqref{pzmax} is negative instead of positive. Then we can rewrite 
$        q(z)=(-p(z))(-\xi)$,  where $-p\in {\mathscr P}^+(k)$.
\end{enumerate}

The new form of asymptotic expansion is the following extension of Definition \ref{Lexpand}.

\begin{definition}\label{zetxp}
Let $\K$ be $\R$ or $\C$, and $(X,\|\cdot\|_X)$ be a normed space over $\K$.
Suppose $g(t)$ is a function from $(T,\infty)$ to $X$ for some $T\in\R$, and $m_*\in \Z_+$. Let $\zeta_*\in \mathscr P_{m_*}^{+}(n_0)$ for some $n_0\ge m_*$. Denote 
\beq\label{Zbar} 
\Xi_k(t)=\left (\widehat \LL_k(t),\zeta_*(\widehat \LL_{n_0}(t))\right ) \text{ for }k\ge -1.
\eeq

\begin{enumerate}[label={\tnum}]
 \item\label{ze1} Let $(\gamma_k)_{k=1}^\infty$ be a divergent, strictly increasing sequence of positive numbers, and $(n_k)_{k=1}^\infty$ be a sequence in $\N$ with $n_k\ge \max\{m_*,n_0\}$ for all $k\ge 1$. 
We say
\beq\label{expan3}
g(t)\sim \sum_{k=1}^\infty p_k(\Xi_{n_k}(t)), \text{ where $p_k\in \widehat {\mathscr P}_{m_*}(n_k,-\gamma_k,X)$ for $k\in\N$, }
\eeq
if, for each $N\in\N$, there is some number $\mu>\gamma_N$ such that
\beq\label{errate3}
\left\|g(t) - \sum_{k=1}^N p_k(\Xi_{n_k}(t))\right\|_X=\bigo(\iln_{m_*}(t)^{-\mu}).
\eeq

\item\label{ze2} Let $N\in\N$, $(\gamma_k)_{k=1}^N$ be positive and strictly increasing, and  $n_*$ be an integer satisfying $n_*\ge \max\{m_*,n_0\}$.
We say
\beq\label{expan4}
g(t)\sim \sum_{k=1}^N p_k(\Xi_{n_*}(t)), \text{ where $p_k\in \widehat {\mathscr P}_{m_*}(n_*,-\gamma_k,X)$ for $1\le k\le N$, }
\eeq
if it holds for all $\mu>\gamma_N$ that
\beq\label{errate4}
\left\|g(t) - \sum_{k=1}^N p_k(\Xi_{n_*}(t))\right\|_X=\bigo(\iln_{m_*}(t)^{-\mu}).
\eeq
\end{enumerate}
\end{definition}

Thanks to \eqref{pluspos}, one has $\zeta_*(\widehat \LL_k(t))>0$ for sufficiently large $t$.
Thus, each term $ p_k(\Xi_{n_k}(t))$ in \eqref{expan3} and \eqref{errate3} is valid for sufficiently large $t$. The same applies to \eqref{expan4} and \eqref{errate4}.

By using the equivalence \eqref{qpequiv}, we have the following equivalent form of \eqref{expan1}
\beq\label{expan5}
g(t)\sim \sum_{k=1}^\infty \widehat p_k(\Xi_{n_k}(t))\iln_{m_*}(t)^{-\gamma_k}, 
\text{ where $\widehat p_k\in \widehat{\mathscr P}_{m_*}(n_k,0,X)$ for $k\in\N$. }
\eeq
For example, when $m_*=0$ the asymptotic expansion \eqref{expan5} reads as
\beq\label{clearpower}
g(t)\sim \sum_{k=1}^\infty \widehat p_k(\Xi_{n_k}(t)) t^{-\gamma_k},
\text{  where $\widehat p_k\in \widehat{\mathscr P}_{0}(n_k,0,X)$ for $k\in\N$. }
\eeq

Suppose $\widehat p_k$ in \eqref{expan5} is of the form 
$$\widehat p_k(z,\zeta)=\sum z^\alpha \zeta^\beta \xi_{\alpha,\beta},\text{ with $\xi_{\alpha,\beta}\ne 0$, as in \eqref{pzedef}.}$$
Thanks to estimate \eqref{plusest} applied to $p=\zeta_*$, we have that each term 
$(z^\alpha \zeta^\beta)|_{(z,\zeta)=\Xi_{n_k}(t)} \xi_{\alpha,\beta} $ of $\widehat p_k(\Xi_{n_k}(t))$ in \eqref{expan5} satisfies, for sufficiently large $t$,
\begin{multline}\label{termest1}
C^{-1} \widehat{\LL}_{n_k}(t)^{\Re(\alpha)} \widehat{\LL}_{n_0}(t)^{\beta \alpha_{\max}} \|\xi_{\alpha,\beta}\|_X
\le \Big \|(z^\alpha \zeta^\beta)|_{(z,\zeta)=\Xi_{n_k}(t)} \xi_{\alpha,\beta} \Big \|_X\\
= \widehat{\LL}_{n_k}(t)^{\Re(\alpha)} \zeta_*(\widehat{\LL}_{n_0}(t))^\beta \|\xi_{\alpha,\beta}\|_X
\le C \widehat{\LL}_{n_k}(t)^{\Re(\alpha)} \widehat{\LL}_{n_0}(t)^{\beta \alpha_{\max}} \|\xi_{\alpha,\beta}\|_X,
\end{multline}
where $C=c^\beta$ when $\beta\ge 0$, and $C=c^{-\beta}$ when $\beta< 0$. 
Above, $\alpha_{\max}$ is defined in Definition \ref{Pplus} for the function $\zeta_*$.
Note also that ${\Re(\alpha)}\in {\mathcal E}_\R(m_*,n_k,0)$ and $\beta \alpha_{\max}\in {\mathcal E}_\R(m_*,n_0,0)$.
Therefore, each term 
$$(z^\alpha \zeta^\beta)|_{(z,\zeta)=\Xi_{n_k}(t)} \xi_{\alpha,\beta} $$  does not contribute  any extra $\iln_{m_*}(t)^r$, with some $r\in\R$, to the decaying mode $\iln_{m_*}(t)^{-\gamma_k}$ in \eqref{expan5}.
This gives a justification to Definition \ref{zetxp} and, hence, also to Definition \ref{Lexpand}.

Combining \eqref{termest1}  with the limit \eqref{LLo}, we obtain
\beq \label{phates}
|\widehat p_k(\Xi_{n_k}(t))|=\bigo(\iln_{m_*}(t)^{\delta}) \text{ for all }\delta>0.    
\eeq
This also implies that the term $p_k(\Xi_{n_k}(t))$ in \eqref{expan4} satisfies
\beqs 
    |p_k(\Xi_{n_k}(t))|=\bigo(\iln_{m_*}(t)^{-\gamma_k+\delta}) \text{ for all }\delta>0.
\eeqs

\section{Main result}\label{mainresult}

In this section, we describe the key result about the asymptotic expansions, as $t\to\infty$, for solutions of equation \eqref{mainode}.
Let the space's dimension $n$ be fixed.

\begin{assumption}\label{assumpA}
 The matrix $A$ is a real $n\times n$ matrix with all eigenvalues having positive real parts.
\end{assumption}

This assumption is very standard for dissipative systems. It guarantees the decay of the associated semigroup $t\in \R\mapsto e^{-tA}$ as $t\to\infty$, see inequality \eqref{eA2} below.
Regarding the nonlinearity in \eqref{mainode}, the function $F$ will be approximated near the origin by functions, not necessarily polynomials,  in the following class.

\begin{definition}\label{phom} 
Suppose $(X,\|\cdot\|_X)$ and $(Y,\|\cdot\|_Y)$ be two real normed spaces.

A function $F:X\to Y$ is positively homogeneous of degree $\beta\ge 0$ if
\beq\label{Fb}
F(tx)=t^\beta F(x)\text{ for any $x\in X$ and any $t>0$.}
\eeq

Define $\mathcal H_\beta(X,Y)$ to be the set of positively homogeneous  functions of order $\beta$ from $X$ to $Y$, and  
denote $\mathcal H_\beta(X)=\mathcal H_\beta(X,X)$.

For a function $F\in \mathcal H_\beta(X,Y)$ that is  bounded on the unit sphere,  define
\beqs
\|F\|_{\mathcal H_\beta}=\sup_{\|x\|_X=1} \|F(x)\|_Y=\sup_{x\ne 0} \frac{\|F(x)\|_Y}{\|x\|_X^\beta}.
\eeqs
\end{definition}

The following are immediate properties.
\begin{enumerate}[label=\rnum]
 \item If $F\in \mathcal H_\beta(X,Y)$ with $\beta>0$, then taking $x=0$ and $t=2$ in \eqref{Fb} gives 
\beq\label{Fzero}  F(0)=0.
\eeq 
If, in addition, $F$ is bounded on the unit sphere, then  
\beq\label{Fhbound}
\|F\|_{\mathcal H_\beta}\in[0,\infty) \text{ and } 
\|F(x)\|_Y\le \|F\|_{\mathcal H_\beta} \|x\|_X^{\beta} \quad \forall x\in X.
\eeq

\item  The zero function (from $X$ to $Y$) belongs to $\mathcal H_\beta(X,Y)$ for all $\beta\ge 0$, and a constant function (from $X$ to $Y$) belongs to $\mathcal H_0(X,Y)$.

\item\label{ps} Each $\mathcal H_\beta(X,Y)$, for $\beta\ge 0$, is a linear space.

\item\label{pm} If $F_1\in \mathcal H_{\beta_1}(X,\R)$ and $F_2\in \mathcal H_{\beta_2}(X,Y)$, then $F_1F_2\in \mathcal H_{\beta_1+\beta_2}(X,Y)$.

\item\label{pp} If $F:X\to Y$ is a (nonzero) homogeneous polynomial of degree $m\in\Z_+$, then $F\in \mathcal H_m(X,Y)$. 
\end{enumerate}

The space $\mathcal H_\beta(X,Y)$ can contain much more complicated functions than homogeneous polynomials, see \cite{CaHK1}. 
We assume the following for the function $F$ in equation \eqref{mainode}.

\begin{assumption}\label{assumpG}  The mapping $F:\R^n\to\R^n$ has the the following properties. 
\begin{enumerate}[label=\tnum]
 \item \label{FL} $F$ is locally Lipschitz on $\R^n$ and $F(0)=0$.
 \item\label{GG} Either \ref{h1} or \ref{h2} below is satisfied.

    \begin{enumerate}[label={\rm (H\arabic*)}]
        \item\label{h1} There exist numbers $\beta_k$, for $k\in\N$, which belong to $(1,\infty)$ and  increase strictly to infinity, and functions $F_k\in \mathcal H_{\beta_k}(\R^n)\cap C^\infty(\R^n\setminus\{0\})$, for $k\in\N$,  such that it holds, for any $N\in\N$, that  
        \beq\label{errF}
            \left|F(x)-\sum_{k=1}^N F_k(x)\right|=\bigo(|x|^{\beta})\text{ as $x\to 0$, for some $\beta>\beta_N$.}
        \eeq

        \item\label{h2} There exist $N_*\in\N$, strictly increasing  numbers $\beta_k$ in $(1,\infty)$, 
        and functions $F_k\in \mathcal H_{\beta_k}(\R^n)\cap C^\infty(\R^n\setminus\{0\})$, for $k=1,2,\ldots,N_*$, such that  
        \beq\label{errF2}
            \left|F(x)-\sum_{k=1}^{N_*} F_k(x)\right|=\bigo(|x|^{\beta})\text{ as  $x\to 0$, for all $\beta>\beta_{N_*}$.}
        \eeq
    \end{enumerate}
\end{enumerate}
\end{assumption}

In Assumption \ref{assumpG}, the condition $F(0)=0$ in part \ref{FL} ensures that the associated homogeneous equation of  \eqref{mainode}, i.e. when $f(t)=0$, has zero as its equilibrium. Moreover, we will conveniently write case \ref{h1} in part \ref{GG} as
\beq\label{Gex}
F(x)\sim \sum_{k=1}^\infty F_k(x),
\eeq
and case \ref{h2} as
\beq\label{Gef}
F(x)\sim \sum_{k=1}^{N_*} F_k(x).
\eeq
The following remarks on Assumption \ref{assumpG} are in order.
\begin{enumerate}[label=\rnum]
 \item Applying \eqref{Fzero} and \eqref{Fhbound} to each  function $F_k$, one has 
\beq \label{Fkb}
    F_k(0)=0,\quad \|F_k\|_{\mathcal H_{\beta_k}}<\infty,\text{ and } |F_k(x)|\le \|F_k\|_{\mathcal H_{\beta_k}}|x|^{\beta_k} \text{ for all $x\in\R^n$.}
\eeq
This makes the remainder estimates in \eqref{errF} and \eqref{errF2} meaningful.

 \item With functions $F_k$ as in \ref{h2}, if $F(x)=\sum_{k=1}^{N_*} F_k(x)$, then $F$ satisfies \eqref{Gef}.

 \item If $F$ is a $C^\infty$-vector field on  $\R^n$ with $F(0)=0$ and $F'(0)=0$, then $F$ satisfies Assumption \ref{assumpG} with the right-hand side of \eqref{Gex}  simply being the Taylor expansion of $F(x)$ about the origin.  
 However, the class of functions $F$ that satisfy Assumption \ref{assumpG} contains much more than these smooth vector fields, see \cite[Section 6]{CaHK1}.
 
 \item We do not require the convergence of the formal series on the right-hand side of \eqref{Gex}.
Even when the convergence occurs, the limit is not necessarily the function $F(x)$. 
\end{enumerate}

By Assumption \ref{assumpG}, for each $N\in\N$ in the case \eqref{Gex}, or $ N\in\N\cap[1, N_*]$ in the case \eqref{Gef}, there is $\varep_N>0$ such that
\beq\label{Ner}
\left|F(x)-\sum_{k=1}^N F_k(x)\right|=\bigo(|x|^{\beta_N+\varep_N})\text{ as } x\to 0.
\eeq
Note from \eqref{Ner} with $N=1$ and the last inequality in \eqref{Fkb} that, as $x\to 0$, 
\beq\label{Gyy}
|F(x)|\le |F_1(x)|+|F(x)-F_1(x)|\le \|F_1\|_{\mathcal H_{\beta_1}}|x|^{\beta_1}+\bigo(|x|^{\beta_1+\varep_1})=\bigo(|x|^{\beta_1}).
\eeq

We  impose conditions on the forcing function $f(t)$ in equation \eqref{mainode} now.

\begin{assumption} \label{fmain}
The function $f$ belongs to $C([T,\infty),\R^n)$, for some $T\ge 0$, and there exists a number $m_*\in\Z_+$ such that $f(t)$  admits the asymptotic expansion, in the sense of Definition \ref{Lexpand} with  $X=\C^n$, 
\beq\label{fas1}
f(t)\sim \sum_{k=1}^\infty p_k(\widehat{\LL}_{\widetilde n_k}(t)), \text{ where $p_k\in \mathscr P_{m_*}(\widetilde n_k,-\gamma_k,\C^n,\R^n)$ for $k\in\N$, }
\eeq
or
\beq\label{fas2}
f(t)\sim \sum_{k=1}^K p_k(\widehat{\LL}_{n_*}(t)), \text{ where $p_k\in \mathscr P_{m_*}(n_*,-\gamma_k,\C^n,\R^n)$ for $1\le k\le K$, }
\eeq
for some integers $K\ge 1$.
\end{assumption}

Denote
\beq \label{n1def}
    n_1=\begin{cases} \widetilde n_1,&\text{ in the case   \eqref{fas1},}\\
                      n_*,&\text{ in the case  \eqref{fas2}.}
     \end{cases}
\eeq
Unlike in the previous work \cite{CaH3,H5}, Assumption \ref{fmain} is not enough and the following additional condition is needed.

\begin{assumption}\label{aspone}
We assume  
\beq \label{pocond}
 p_1(z)=p_0(z) z_{m_*}^{-\gamma_1}\xi_0,\text{ where }p_0 \in {\mathscr P}_{m_*}^+(n_1),\ \xi_0\in \R^n\setminus\{0\},
 \eeq 
 and $z=(z_{-1},z_0,\ldots,z_{n_1})\in (0,\infty)^{n_1+2}$.
 \end{assumption}

Consider the ODE system \eqref{mainode} in $\R^n$ and a solution $y(t)\in\R^n$ such that
\beq \label{soln}
\text{$y\in C^1([T_0,\infty))$  satisfying equation \eqref{mainode} on $(T_0,\infty)$, for some $T_0>0$,}    
\eeq
and 
\beq \label{decay}
    \lim_{t\to\infty} y(t)=0.
\eeq
The main result of the paper is the following.

\begin{theorem}[Main Theorem]\label{mainthm}
Under Assumptions  \ref{assumpA}, \ref{assumpG}, \ref{fmain} and \ref{aspone},
 there exist a divergent, strictly increasing sequence $(\mu_k)_{k=1}^\infty\subset (0,\infty)$,
an increasing sequence $(n_k)_{k=0}^\infty\subset \Z_+\cap[m_*,\infty)$, a function
\beq \label{zestacond}
    \zeta_*\in \mathscr P_{m_*}^+(n_0),
\eeq 
and functions
 \beq \label{qkhc}
 q_k\in \widehat{\mathscr P}_{m_*}(n_k,-\mu_k,\C^n,\R^n)\text{ for all $k\in\N$,}
 \eeq 
  such that any solution  $y(t)$  of \eqref{mainode} as in \eqref{soln}  and \eqref{decay} admits the asymptotic expansion  
   \beq \label{solnxp}
  y(t)\sim \sum_{k=1}^\infty q_k(\Xi_{n_k}(t)) \text{ in the sense of Definition \ref{zetxp},}
  \eeq
    where $\Xi_k(t)$ is defined by \eqref{Zbar}.
  \end{theorem}

In fact, Theorem \ref{mainthm} is a concise version of Theorem \ref{mainthm2} in section \ref{asympsec} below, where explicit constructions of the functions $\zeta_*$ and $q_k$ will be given.
Some examples will be given in Example \ref{examples} below.

\begin{remark} 
We compare the above result with recent work on the subject.
The paper \cite{CaHK1} requires the matrix  $A$  to be diagonalizable with positive eigenvalues which is stricter than Assumption \ref{assumpA}. Moreover, it treats the case $f(t)=0$ only.
The paper \cite{H5} requires the function $F$ to have the approximation \eqref{Gex} with each $F_k$ being a homogeneous polynomial of degree $k+1$. Such a condition is  much more than what we need in Assumption \ref{assumpG}. 
However, we, in turn, require an additional  Assumption \ref{aspone} for the forcing function $f(t)$.
The paper \cite{H6} deals with the NSE \eqref{NSE} which is a system of PDE not ODE, but the nonlinear term $B(y,y)$ is a bilinear form, not a general function $F(y)$.
 \end{remark}

\section{Fundamental issues}\label{tools}

We assume throughout this section the following Assumption \ref{complexA}.

  \begin{assumption}\label{complexA}
  Let $A$ be a complex $n\times n$ matrix with all eigenvalues having positive real parts.
\end{assumption}

Denote by $\lambda_1$ the minimum of the real parts of the eigenvalues of $A$. Then, for any $\lambda\in(0,\lambda_1)$, there is $C_\lambda>0$ such that
\beq\label{eA2}
\left |e^{-tA}\right|\le C_\lambda  e^{-\lambda t}\quad \text{ for all } t\ge 0.
\eeq
Indeed, estimate \eqref{eA2} comes from the fact that, thanks to the Jordan normal form of $A$, the exponential $e^{-tA}$ is a finite sum of the terms of the form $t^m e^{-\Lambda t}M$, where  $m\in\Z_+$, $\Lambda$ is an an eigenvalue of $A$ and $M$ is a complex $n\times n$  matrix.

\subsection{Important linear operators}
The following linear operators are crucial in the construction of $q_k$ in Theorem \ref{mainthm}.

\begin{definition}\label{newop}
Given an integer $k\ge -1$, let $p\in \widehat{\mathscr P}(k,\C^n)$ be given by \eqref{pzdef} with   $\K=\C$, $X=\C^n$, and $z$ and $\alpha$ as in \eqref{azvec}. 

Define, for $j=-1,0,\ldots,k$, the function $\mathcal M_jp:(0,\infty)^{k+2}\times(0,\infty)\to \C^n$ by 
\beq\label{MM}
(\mathcal M_jp)(z,\zeta)=\sum_{(\alpha,\beta)\in S} \alpha_j z^\alpha \zeta^\beta \xi_{\alpha,\beta}.
\eeq 

In the case $k\ge 0$, define the function $ \mathcal R  p:(0,\infty)^{k+2}\times(0,\infty)\to \C^n$ by 
 \beq\label{chiz}
 (\mathcal R p)(z,\zeta)=
  \sum_{j=0}^k z_0^{-1}z_1^{-1}\ldots z_{j}^{-1}(\mathcal M_j p)(z,\zeta).
 \eeq 

In the case $p\in \widehat{\mathscr P}_{-1}(k,0,\C^n)$, define the function $\mathcal Z_Ap:(0,\infty)^{k+2}\times(0,\infty)\to \C^n$ by 
 \beq\label{ZAp}
 (\mathcal Z_Ap)(z,\zeta)=\sum_{(\alpha,\beta)\in S}  z^{\alpha} \zeta^\beta (A+\alpha_{-1}I_n)^{-1} \xi_{\alpha,\beta}.
 \eeq
\end{definition}

In particular, when $j=-1$, \eqref{MM} reads as
\beq \label{Mminus}
    \mathcal M_{-1}p(z,\zeta)=\sum_{(\alpha,\beta)\in S} \alpha_{-1} z^\alpha\zeta^\beta \xi_{\alpha,\beta}.
\eeq

An equivalent definition of $(\mathcal Rp)(z,\zeta)$ in \eqref{chiz} is
\beqs
 (\mathcal R p)(z,\zeta)=\frac{\partial p(z,\zeta)}{\partial z_0}+
  \sum_{j=1}^k z_0^{-1}z_1^{-1}\ldots z_{j-1}^{-1}\frac{\partial p(z,\zeta)}{\partial z_j}.
\eeqs 

Thanks to \eqref{Eminus}, one has  $\Re(\alpha_{-1})=0$ in \eqref{ZAp}  for all $(\alpha,\beta)\in S $.
Combining this with Assumption \ref{complexA}, we have  $A+\alpha_{-1}I_n$ is invertible and definition \eqref{ZAp} is valid.
From \eqref{ZAp}, one has $\mathcal Z_Ap\in \widehat{\mathscr P}_{-1}(k,0,\C^n)$. Clearly,
\beq\label{ZAM}
 (A+\mathcal M_{-1})(\mathcal Z_Ap)=\mathcal Z_A((A+\mathcal M_{-1})p)=p\quad \forall p\in \mathscr P_{-1}(k,0,\C^n).
 \eeq
 
 By the mappings $p\mapsto \mathcal M_j p$, $p\mapsto \mathcal Rp$, and $p\mapsto \mathcal Z_A p$, one can define
 linear operator $\mathcal M_j$, for $-1\le j\le k$, on $\widehat{\mathscr P}(k,\C^n)$,
  linear operator $\mathcal R$ on $\widehat{\mathscr P}(k,\C^n)$ for $k\ge 0$,
 and linear operator  $\mathcal Z_A$ on $\widehat{\mathscr P}_{-1}(k,0,\C^n)$ for $k\ge -1$.
In fact, the operators $\mathcal M_j$, $\mathcal R$, $\mathcal Z_A$  only act on the $z$ variable.
If $p\in \mathscr P(k,\C^n)$, then $p=p(z)$ and these operators  are exactly the same as in \cite[Definition 5.2]{H5}.
The following properties of $\mathcal M_j$, $\mathcal R$, $\mathcal Z_A$ are obvious from their definitions.

\begin{lemma}\label{MRZgen}
Let $\tilde{\mathscr P}=\mathscr P$ or $\widehat{\mathscr P}$. One has the following.
\begin{enumerate}[label=\rnum]
\item \label{R0} For $k\ge m\ge 0$ and $\mu\in\R$,   if $p$ is in $\tilde{\mathscr P}_{m}(k,\mu,\C^n)$,  then all $\mathcal M_jp$, for $-1\le j\le k$,  and $\mathcal Z_Ap$  are also in $\tilde{\mathscr P}_{m}(k,\mu,\C^n)$. 

\item\label{R1} 
 $\mathcal R p(z)$ has the same powers of  $z_{-1}$ as $p(z)$.
 
\item \label{R2}   
If $p\in\tilde{\mathscr P}_0(k,\mu,\C^n)$, then $\mathcal R p\in\tilde{\mathscr P}_0(k,\mu-1,\C^n)$.
\end{enumerate}
\end{lemma}

Note in \ref{R0} of Lemma \ref{MRZgen} that $\mathcal Z_Ap$  is validly defined thanks to the inclusion \eqref{Pm10}.
Based on Lemma \ref{MRZgen}, we can assert further properties of the mappings $\mathcal M_j$, $\mathcal R$ and $\mathcal Z_A$ when restricted to the classes in Definition \ref{realF}.

\begin{lemma}
\label{invar2}
Let $\tilde{\mathscr P}=\mathscr P$ or $\widehat{\mathscr P}$. 
Below, $\mathcal Z_A$ is considered only under Assumption \ref{assumpA}.

The following statements hold true.
\begin{enumerate}[label=\tnum]
\item\label{invi}  
	Each $\mathcal M_j$, for $-1\le j\le k$, maps $\tilde{\mathscr P}(k,\C^n,\R^n)$ into itself, 
	  $\mathcal R$ maps $\tilde{\mathscr P}(k,\C^n,\R^n)$, for $k\ge 0$, into itself, 
	and   $\mathcal Z_A$ maps $\tilde{\mathscr P}_{-1}(k,0,\C^n,\R^n)$, for $k\ge -1$,  into itself. 
\item\label{invii} 
	All $\mathcal M_j$, for $-1\le j\le k$,  and $\mathcal Z_A$ 
map  $\tilde{\mathscr P}_{m}(k,\mu,\C^n,\R^n)$ into itself for any integers $k\ge m\ge 0$ and real number $\mu$.
\item \label{inviii} 
	 $\mathcal R$ 
maps  $\tilde{\mathscr P}_{0}(k,\mu,\C^n,\R^n)$ into $\tilde{\mathscr P}_{0}(k,\mu-1,\C^n,\R^n)$ for any  $k\in\Z_+$ and $\mu\in\R$. 
\end{enumerate}
\end{lemma}
\begin{proof}
For the case $\tilde{\mathscr P}={\mathscr P}$, this lemma is Lemma 10.7 of \cite{H5}.
    For the case $\tilde{\mathscr P}=\widehat{\mathscr P}$, the proof is the same thanks to the fact $\zeta>0$ and $\beta\in\R$ in \eqref{pzedef}. 
\end{proof}

\subsection{Asymptotic approximation for linear equations}
We recall here two relevant results from \cite{H5} that deal with functions in Definition \ref{Fclass}.

\begin{lemma}[{\cite[Lemma 5.6]{H5}}]\label{logode}
If $k\in \Z_+$ and $q\in \mathscr P(k,\C^n)$, then
 \beq\label{dq0}
\ddt q(\widehat \LL_k(t))=\mathcal M_{-1} q(\widehat \LL_k(t))+\mathcal Rq(\widehat \LL_k(t)) \text{ for }t>E_k(0).
 \eeq
In particular, when  $k\ge m\ge 1$, $\mu\in\R$, and $q\in\mathscr P_{m}(k,\mu,\C^n)$, one has
 \beqs
 \ddt q(\widehat \LL_k(t))=\mathcal M_{-1} q(\widehat \LL_k(t))+\bigo(t^{-\gamma})\quad
 \text{for all }\gamma\in(0,1).
 \eeqs
 \end{lemma}

\begin{theorem}[{\cite[Theorem 5.5]{H5}}]\label{iterlog}
Given integers $m,k\in \Z_+$ with $k\ge m$, and a number $t_0>E_k(0)$.
Let  $\mu>0$, $p\in \mathscr P_{m}(k,-\mu,\C^n)$, and let function $g\in C([t_0,\infty),\C^n)$ satisfy 
\beqs 
|g(t)|=\bigo(\iln_m(t)^{-\alpha})\text{ for some }\alpha>\mu.
\eeqs 

Suppose $y\in C([t_0,\infty),\C^n)$ is a solution of 
 \beqs
 y'=-Ay+p(\widehat \LL_k(t))+g(t)\text{ on } (t_0,\infty).
 \eeqs

 Then  there exists  $\delta >0$  such that 
 \beq\label{yZp}
 |y(t)-\mathcal Z_Ap(\widehat \LL_k(t))|=\bigo(\iln_m(t)^{-\mu-\delta}).
 \eeq
 \end{theorem}

A new version of Lemma \ref{logode} for the new functions in Definitions \ref{Fsubo} and \ref{Pplus} is presented next.

 \begin{lemma}\label{newdq}
 Assume 
$k\ge m\ge 0$,  $\mu\in\R$, $q\in\widehat{\mathscr P}_{m}(k,\mu,\C^n)$ and $\zeta_*\in \mathscr P_m^+(k)$.
 Denote 
 \beq \label{Xikbar}
     \Xi_k(t)=\left (\widehat \LL_k(t),\zeta_*(\widehat \LL_{k}(t))\right ).
 \eeq
Then
 \beq\label{dqze0}
\ddt q(\Xi_k(t))
=\left.\left\{ \mathcal M_{-1} q(z,\zeta)+\mathcal Rq(z,\zeta) +\mathcal R\zeta_*(z)\frac{\partial q(z,\zeta)}{\partial\zeta} \right\} \right|_{(z,\zeta)=\Xi_k(t)}.
 \eeq

If $m=0$, then 
 \beq\label{dqze1}
\ddt q(\Xi_k(t))
=\mathcal M_{-1} q(\Xi_k(t))+\bigo(t^{-\mu-\gamma})\text{ for all } \gamma\in(0,1).
 \eeq

If  $m\ge 1$, then
 \beq\label{dqze2}
 \ddt q(\Xi_k(t))=\mathcal M_{-1} q(\Xi_k(t)) +\bigo(t^{-\gamma})
 \text{ for all }\gamma\in(0,1).
 \eeq
 \end{lemma}
\begin{proof}
Note from \eqref{strictPp} and \eqref{Mminus} that $\mathcal M_{-1} \zeta_*=0$.
By the Chain Rule and applying formula \eqref{dq0} twice, we obtain
\begin{align*}
&\ddt q(\Xi_k(t))
=\mathcal M_{-1} q(\Xi_k(t))+\mathcal Rq(\Xi_k(t))
+\frac{\partial q}{\partial\zeta}(\Xi_k(t)) \ddt \zeta_*(\widehat \LL_k(t))\\
&=\left.\left\{ \mathcal M_{-1} q(z,\zeta)+\mathcal Rq(z,\zeta)
+\frac{\partial q(z,\zeta)}{\partial\zeta} \Big  (\mathcal M_{-1} \zeta_*(z)+\mathcal R\zeta_*(z)\Big)\right\}\right|_{(z,\zeta)=(\Xi_k(t))},
\end{align*}
which yields \eqref{dqze0}.

Consider $m=0$. Applying estimate \eqref{plusrate} to the function $p=\zeta_*$ yields  
$ \zeta_*(\widehat{\LL}_k(t))^s=\bigo(t^\delta)$ for any numbers  $s\in\R$ and $\delta>0$.
Then one has
\begin{align*}
\left|\frac{\partial q}{\partial\zeta}(\Xi_k(t))\right|
&=\bigo(t^{-\mu+\varep})
\ \forall \varep>0,\\
\intertext{and, by Lemma \ref{MRZgen}\ref{R2},}
\left|\mathcal Rq(\Xi_k(t))\right|&=\bigo(t^{-\mu-\gamma}),\quad 
|\mathcal R\zeta_*(\widehat \LL_k(t))|
=\bigo(t^{-\gamma}) \ \forall \gamma\in(0,1).
\end{align*}
Combining these estimates with \eqref{dqze0}, we obtain \eqref{dqze1}.

Consider $m\ge 1$. We similarly have
\begin{align*}
\left|\frac{\partial q}{\partial\zeta}(\Xi_k(t))\right |&=\bigo(t^{\varep}) \ \forall \varep>0,\\
|\mathcal Rq(\Xi_k(t))|,\ 
|\mathcal R\zeta_*(\widehat \LL_k(t))|&=\bigo(t^{-\gamma}) \ \forall \gamma\in(0,1).
\end{align*}
Combining these estimates with \eqref{dqze0}, we obtain \eqref{dqze2}.
\end{proof}

The following lemma  is an updated version of  \cite[Lemma 5.4]{H5} and will be used to prove the fundamental approximation result  -- Theorem \ref{newlin}.

\begin{lemma}\label{newplem} 
Let $m\in Z_+$, $\mu>0$, $k\ge m$, $p\in\widehat{\mathscr P}_{m}(k,-\mu,\C^n)$, $\zeta_*\in{\mathscr P}_m^+(k)$ and a sufficiently large number $T_*>E_k(0)$.
Denote $\Xi_k(t)$ as in \eqref{Xikbar}.

If $m=0$, then one has,  for any $\gamma\in(0,1)$,
 \beq\label{log1}
\int_0^t e^{-(t-\tau)A} p(\Xi_k(T_*+\tau)) \d\tau
=(\mathcal Z_A p)(\Xi_k(T_*+t))
 +\bigo((T_*+t)^{-\mu-\gamma}) .
 \eeq

If $m\ge1$, then one has,  for any $\gamma\in(0,1)$,
 \beq\label{log2}
 \int_0^t e^{-(t-\tau)A} p(\Xi_k(T_*+\tau)) \d\tau =(\mathcal Z_A p)(\Xi_k(T_*+t)) 
 +\bigo((T_*+t)^{-\gamma}).
 \eeq
 \end{lemma}
\begin{proof}
It suffices to prove \eqref{log1} and \eqref{log2} for 
$p(z,\zeta)= z^{ \alpha}\zeta^\beta\xi$, with 
$$z=(z_{-1},z_0,\ldots,z_k)\in(0,\infty)^{k+2},\
\alpha=(\alpha_{-1},\alpha_0,\ldots,\alpha_k)\in\mathcal E_\C(m,k,-\mu),\ 
\beta\in \R,
$$
and some $\xi\in \C^n$.
Since $m\ge 0$, we have $\Re(\alpha_{-1})=0$.
Let 
\beq \label{alprim}
    \alpha'=(0,\alpha_0,\ldots,\alpha_k)\text{ and } q(z,\zeta)=z^{\alpha'}\zeta^\beta\xi.
\eeq

By defining  $F(t)=q(\Xi_k(T_*+t))$, we can write 
 \beqs
 p(\Xi_k(T_*+t)) =e^{\alpha_{-1} (T_*+t)}F(t).
 \eeqs 
 
Denote $I(t)=\int_0^t e^{-(t-\tau)A} p(\Xi_k(T_*+\tau)) \d\tau$. Then
 \begin{align*}
 I(t)= \int_0^t e^{\alpha_{-1}(T_*+ \tau)}e^{-(t-\tau)A}F(\tau)   \d\tau 
= e^{\alpha_{-1} (T_*+t)}\int_0^t e^{-(t-\tau)(A+\alpha_{-1} I_n)}F(\tau)    \d\tau.
 \end{align*}
Using the fact 
$$\frac{\d}{\d\tau} \left[(A+\alpha_{-1} I_n)^{-1} e^{-(t-\tau)(A+\alpha_{-1} I_n)}\right]=e^{-(t-\tau)(A+\alpha_{-1} I_n)},$$
we apply integration by parts to obtain 
  \begin{align*}
  I(t) 
&=  e^{\alpha_{-1} (T_*+t)}\left\{(A+\alpha_{-1} I_n)^{-1} e^{-(t-\tau)(A+\alpha_{-1} I_n)}F(\tau)   \Big|_{\tau=0}^{\tau=t}\right.\\
&\qquad\qquad\qquad \left.- \int_0^t (A+\alpha_{-1} I_n)^{-1}e^{-(t-\tau)(A+\alpha_{-1} I_n)} \frac{\d F(\tau)}{\d \tau}  \d\tau \right\}\\
&=e^{\alpha_{-1} (T_*+t)}(A+\alpha_{-1} I_n)^{-1}F(t) -e^{\alpha_{-1} T_*}(A+\alpha_{-1} I_n)^{-1} e^{-tA}F(0)  -J(t),
 \end{align*} 
where
\beqs
J(t)=\int_0^t e^{\alpha_{-1}(T_*+\tau)}(A+\alpha_{-1} I_n)^{-1}e^{-(t-\tau)A} \frac{\d F(\tau)}{\d \tau}  \d\tau.
\eeqs
Hence,
\beq\label{Iint}
  I(t) = (A+\alpha_{-1} I_n)^{-1}p(\Xi_k(T_*+t)) -(A+\alpha_{-1} I_n)^{-1}e^{-tA} p(\Xi_k(T_*))   -J(t).
\eeq

We consider each term on the right-hand side of \eqref{Iint}.
For the first term, it is obvious that
\beq\label{Iz}
(A+\alpha_{-1} I_n)^{-1}p(\Xi_k(T_*+t)) =(\mathcal Z_Ap)(\Xi_k(T_*+t)) .
\eeq
For the second term, one has, thanks to \eqref{eA2},  
\beq\label{Iini}
(A+\alpha_{-1} I_n)^{-1}e^{-tA} p(\Xi_k(T_*))=\bigo(e^{-\lambda_1 t/2}).
\eeq
For the third term,  we estimate $J(t)$ by noticing that $|e^{-\alpha_{-1}(T_*+t)}|=1$ and applying inequality \eqref{eA2} again to have 
\beq\label{JJ}
|J(t)|\le C_1\int_0^t e^{-\lambda_1(t-\tau)/2} \left|\frac{\d F(\tau)}{\d \tau} \right| \d\tau,\text{ where }
C_1=C_0|(A+\alpha_{-1} I_n)^{-1}|.
\eeq

Let $\gamma$ be an arbitrary number in $(0,1)$. Thanks to  \eqref{Mminus} and \eqref{alprim}, one has
\beq \label{Mqz}
    \mathcal M_{-1}q=0.
\eeq 

\medskip
Consider $m=0$. 
By \eqref{dqze1} and \eqref{Mqz},
\beqs
\left| \frac{\d F(t)}{\d t}\right|=\bigo ((T_*+t)^{-\mu-\gamma}).
 \eeqs
By the continuity of the functions $t\mapsto \d F(t)/\d t$ and $t\mapsto  (T_*+t)^{-\mu-\gamma}$ on $[0,\infty)$, we deduce
\beq\label{dF1}
\left| \frac{\d F(t)}{\d t}\right|\le C_2 (T_*+t)^{-\mu-\gamma} \text{ for all }t\ge 0,
 \eeq
for some positive constant $C_2$.
Using \eqref{dF1} in \eqref{JJ} and then applying inequality \eqref{iine2} to estimate the resulting  integral give
\beq\label{Jbo}
|J(t)|\le C_1C_2\int_0^t e^{-\lambda_1(t-\tau)/2}(T_*+\tau)^{-\mu-\gamma} \d\tau
 =\bigo((T_*+t)^{-\mu-\gamma}).
\eeq
Combining \eqref{Iint},  \eqref{Iz}, \eqref{Iini} and \eqref{Jbo} yields \eqref{log1}.

\medskip
Consider $m\ge 1$. By \eqref{dqze2} and \eqref{Mqz},
\beqs
\left|\frac{\d F(t)}{\d t}\right|=\bigo((T_*+t)^{-\gamma}).
\eeqs 
We obtain, similar to \eqref{Jbo}, that
\beq\label{Jbo2}
|J(t)|=\bigo((T_*+t)^{-\gamma}).
\eeq
Combining \eqref{Iint},  \eqref{Iz}, \eqref{Iini} and \eqref{Jbo2} yields \eqref{log2}.
 \end{proof}

The next theorem is a new version of Theorem \ref{iterlog} and is essential in proving the Theorem \ref{mainthm}.

 \begin{theorem}\label{newlin}
Given integers $m,k\in \Z_+$ with $k\ge m$, and a number  $\mu>0$.
Assume $\zeta_*\in\mathscr P_{m}^+(k)$,   $p\in \widehat{\mathscr P}_{m}(k,-\mu,\C^n)$, and function $g\in C([t_0,\infty),\C^n)$, for some sufficiently large number $t_0>E_k(0)$, satisfies
\beq \label{gbio}
|g(t)|=\bigo(\iln_m(t)^{-\alpha})\text{ for some }\alpha>\mu.
\eeq 

Suppose $y\in C^1([t_0,\infty),\C^n)$ is a solution of 
 \beq\label{yeq4}
 y'=-Ay+p(\widehat \LL_k(t),\zeta_*(\widehat \LL_k(t)))+g(t)\text{ on } (t_0,\infty).
 \eeq

 Then  there exists  $\delta >0$  such that 
 \beq\label{ypL}
\left |y(t)-\mathcal Z_Ap(\widehat \LL_k(t),\zeta_*(\widehat \LL_k(t)))\right|=\bigo(\iln_m(t)^{-\mu-\delta}).
 \eeq
 \end{theorem}
  \begin{proof}
  Regarding equation \eqref{yeq4}, we implicitly assumed that $t_0$ is sufficiently large such that
  $$\zeta_*(\widehat \LL_k(t))>0 \text{ for all }t\in [t_0,\infty).$$
Denote $\Xi_k(t)$ as in \eqref{Xikbar}. By the variation of constant formula, 
\beq\label{ytt}
y(t_0+t)=e^{-tA}y(t_0)+\int_{0}^{t} e^{-(t-\tau)A} p(\Xi_k(t_0+\tau))\d\tau+\int_{0}^{t} e^{-(t-\tau)A}g(t_0+\tau)\d\tau.
\eeq

For the first term on the right-hand side of \eqref{ytt}, we have, by \eqref{eA2},  
\beqs 
e^{-tA}y(t_0)=\bigo(e^{-\lambda_1 t/2})=\bigo(\iln_m(t_0+t)^{-\mu-1}).
\eeqs 

For the second term on the right-hand side of \eqref{ytt},  we apply Lemma \ref{newplem} to $T_*=t_0$, using identity \eqref{log1} when $m=0$, and identity  \eqref{log2} when $m\ge 1$, both with $\gamma=1/2$.  
In the latter case $m\ge 1$, we also estimate $\bigo((t_0+t)^{-1/2})=\bigo(\iln_m(t_0+t)^{-\mu-1/2})$.
These result in 
\begin{align*}
   \int_{0}^{t} e^{-(t-\tau)A} p(\Xi_k(t_0+\tau))\d\tau
 =(\mathcal Z_Ap)(\Xi_k(t_0+t))+\bigo(\iln_m(t_0+t)^{-\mu-1/2}).
\end{align*}

Consider  the third term on the right-hand side of \eqref{ytt}.
By \eqref{gbio} and the continuity of $g(t)$ and $\iln_m(t)>0$ on $[t_0,\infty)$, there is $C_1>0$ such that 
\beq\label{gbound}
|g(t)|\le C_1\iln_m(t)^{-\alpha}\quad\text{ for all $t\ge t_0$. }
\eeq
Then combining  inequalities \eqref{eA2}, \eqref{gbound} with \eqref{iine2} yields
\beqs
\left| \int_{0}^{t} e^{-(t-\tau)A}g(t_0+\tau)\d\tau\right|
\le C_1C_2\int_{0}^{t} e^{-\lambda_1(t-\tau)/2}\iln_m(t_0+t)^{-\alpha}\d\tau \le C_3\iln_m(t_0+t)^{-\alpha}
\eeqs
for all $t\ge 0$ and some constants $C_2,C_3>0$.
Let $\delta=\min\{1/2,\alpha-\mu\}>0$. We obtain from \eqref{ytt} and the above calculations and estimates  that
\beqs
y(t_0+t)=(\mathcal Z_Ap)(\Xi_k(t_0+t))+\bigo(\iln_m(t_0+t)^{-\mu-\delta}),
\eeqs
which proves \eqref{ypL}.
\end{proof}

The estimate \eqref{ypL} shows that $(\mathcal Z_Ap)(\widehat \LL_k(t),\zeta_*(\widehat \LL_k(t)))$ is an asymptotic approximation of $y(t)$.
This is the building block for the recursive construction of the functions $q_k$ in Theorem \ref{mainthm}, see the proof of Theorem \ref{mainthm2} below.

\subsection{Complexification of linear and multi-linear mappings}\label{comfysec}
The proof of Theorem \ref{mainthm} will make use of the following complexification.

\begin{definition}\label{cmplxify}
Let $\{e_j:1\le j\le n\}$ be the canonical basis for $\R^n$ and $\C^n$.
Let $m\ge 1$ and $M$ be an $m$-linear mapping from $(\R^n)^m$ to $\R^p$.
The complexification of $M$ is  $M_\C:(\C^n)^m\to \C^p$  defined by
\beq\label{MCdef}
\begin{aligned}
&M_\C\left(\sum_{j_1=1}^n z_{1,j_1} e_{j_1}, \sum_{j_2=1}^n z_{1,j_2} e_{j_2},\ldots, \sum_{j_m=1}^n z_{m,j_m} e_{j_m}\right )\\
& =\sum_{j_1,j_2,\ldots,j_m=1}^n  z_{1,j_1}z_{2,j_1} \ldots  z_{m,j_m}  M( e_{j_1},e_{j_2}, \ldots,  e_{j_m} )
\end{aligned}
\eeq
for all $z_{k,j_k}\in\C$, $k=1,2,\ldots,m$ and $j_k=1,2,\ldots,n$.
\end{definition}

Then $M_\C$ is the unique extension of $M$ to an $m$-linear mapping (over $\C$) from $(\C^n)^m$ to $\C^p$.
Because $M( e_{j_1},e_{j_2}, \ldots,  e_{j_m} )$ in \eqref{MCdef} are $\R^p$-valued, one has the following conjugation property
\beq\label{MCbar}
M_\C(\bar \xi_1,\bar \xi_2,\ldots,\bar \xi_m)=\overline{M_\C(\xi_1,\xi_2,\ldots,\xi_m)} \quad\forall \xi_1,\xi_2,\ldots,\xi_m\in\C^n .
\eeq

\begin{lemma}\label{prodclass}
    Let $s\ge 1$ and $M:\C^{m_1}\times \ldots\times \C^{m_s}\to \C^{m_0}$ be an $s$-linear mapping such that
    \beq \label{Mbicond}
        \overline{M(\xi_1,\ldots,\xi_s)}=M(\bar \xi_1,\ldots,\bar\xi_s)\text{ for all }\xi_j\in\C^{m_j}, \ 1\le j\le s.
    \eeq
Let $\tilde {\mathscr P}={\mathscr P}$ or $\widehat{\mathscr P}$.   Then one has, for any $k\ge m\ge -1$ and $p_j\in \tilde{\mathscr P}_{m}(k,\mu_j,\C^{m_j},\R^{m_j})$, for $j=1,\ldots,s$,   that
    \beq \label{Mpq}
        M(p_1,\ldots,p_s)\in \tilde{\mathscr P}_{m}(k,\mu_1+\ldots+\mu_s,\C^{m_0},\R^{m_0}).
    \eeq    
\end{lemma}
\begin{proof}
Consider the case $\tilde{\mathscr P}=\widehat{\mathscr P}$.
By the virtue of \eqref{phalf}, it suffices to prove \eqref{Mpq} for 
$$p_j=q_j+\bar q_j \text{ for $1\le j\le s$, where } q_j(z,\zeta)=z^{\alpha_{(j)}} \zeta^{\beta_j}\xi_j.$$
Denote $\kappa_1(v)=v$ and $\kappa_{-1}(v)=\bar v$, for $v\in\C^r$. For $\delta\in\{1,-1\}$, one has 
    $$\overline{\kappa_\delta(v)}=\kappa_{-\delta}(v)\text{ and  }\overline{z^{\kappa_\delta(\alpha)}}=z^{\overline{\kappa_\delta(\alpha)}}=z^{\kappa_{-\delta}(\alpha)}.$$
    Thanks to property \eqref{Mbicond}, we have
    \beq \label{Mex}
    \begin{aligned}
        M(p_1,\ldots,p_s)
        &=M(q_1+\bar q_1,\ldots,q_s+\bar q_s)
        =\sum_{j_1,\ldots,j_s\in\{1,-1\}}  M(\kappa_{j_1}(q_1),\ldots,\kappa_{j_s}(q_s))\\
        &=\sum_{\substack{j_1,\ldots,j_s\in\{1,-1\},\\ (j_1,\ldots,j_s)>\mathbf 0_s}}  \big[M(\kappa_{j_1}(q_1),\ldots,\kappa_{j_s}(q_s))
        +M(\kappa_{-j_1}(q_1),\ldots,\kappa_{-j_s}(q_s))\big]\\
        &=\sum_{\substack{j_1,\ldots,j_s\in\{1,-1\},\\ (j_1,\ldots,j_s)>\mathbf 0_s}} \big[ M(\kappa_{j_1}(q_1),\ldots,\kappa_{j_s}(q_s))+\overline{M(\kappa_{j_1}(q_1),\ldots,\kappa_{j_s}(q_s))}\big].
    \end{aligned}
    \eeq
Above, $\mathbf 0_s$ is the zero vector in $\R^s$ and we used the lexicography order for the vector $(j_1,\ldots,j_s)$.
    We compute 
    \beq \label{Mpqo}
       M(\kappa_{j_1}(q_1),\ldots,\kappa_{j_s}(q_s))
       =z^{\kappa_{j_1}(\alpha_{(1)})+\ldots+\kappa_{j_s}(\alpha_{(s)})} 
       \zeta^{\beta_1+\ldots+\beta_s}M(\kappa_{j_1}(\xi_1),\ldots,\kappa_{j_s}(\xi_s)).
    \eeq
    Observe that $\kappa_{j_1}(\alpha_{(1)})+\ldots+\kappa_{j_s}(\alpha_{(s)}) \in \mathcal E_\C(m,k,\mu_1+\ldots+\mu_s)$.
        Combining this fact with \eqref{Mex} and \eqref{Mpqo}, we obtain \eqref{Mpq}.

The proof for the case $\tilde{\mathscr P}={\mathscr P}$ is the same by dropping the variable $\zeta$.      
\end{proof}

\begin{corollary}\label{cplxcor}
Let $\tilde {\mathscr P}={\mathscr P}$ or $\widehat{\mathscr P}$. 
For $k\ge m\ge -1$, if $p\in \tilde{\mathscr P}_{m}(k,\mu_1,\C,\R)$ and 
$q\in \tilde{\mathscr P}_{m}(k,\mu_2,\C^{n},\R^{n})$, then
     $pq\in \tilde{\mathscr P}_{m}(k,\mu_1+\mu_2,\C^{n},\R^{n})$.
\end{corollary}
\begin{proof}
This is a consequence of Lemma \ref{prodclass} applied to $s=2$, $m_1=1$, $m_2=m_0=n$, and $M(x,y)=xy$ for $x\in \C$ and $y\in\C^n$.   
\end{proof}

\section{The asymptotic behavior of solutions}\label{asympsec}

We come to the main task of establishing the asymptotic behavior, as $t\to\infty$, for solutions of equation \eqref{mainode}.

\subsection{The first asymptotic approximation}\label{firstapprx}
In this subsection, we consider equation \eqref{mainode} in $\C^n$.

\begin{assumption}\label{complexFf}
    The function $F:\C^n\to\C^n$ is continuous and satisfies 
    \beq \label{Fbe}
    |F(x)|=\bigo(|x|^\beta), \text{ as $x\to 0$,  for some $\beta>1$.}
\eeq

The function $f:[T,\infty)\to \C^n$, for some $T\ge 0$, is continuous. 
\end{assumption}

\begin{theorem}\label{thmap1}
Suppose Assumptions \ref{complexA} and  \ref{complexFf} hold true.
Assume there exists a function $p\in\mathscr P_{m}(k,-\mu,\C^n)$, for some integers $k\ge m\ge 0$ and positive number $\mu$, such that 
\beq\label{fperr}
|f(t)-p(\widehat\LL_k(t))|=\bigo(\iln_m(t)^{-\mu-\delta_0})\text{ as $t\to\infty$,}
\eeq
for some number $\delta_0>0$.  Let $y(t)\in\C^n$ be a solution of \eqref{mainode} as in \eqref{soln} and \eqref{decay}.
Then there is $\varep>0$ such that 
\beq\label{abehav}
|y(t)-(\mathcal Z_A p)(\widehat\LL_k(t))|=\bigo(\psi(t)^{-\mu-\varep})\text{ as $t\to\infty$.}
\eeq
\end{theorem}
\begin{proof}
Let $\delta$ be any positive number such that $\delta<\min\{\mu,\delta_0\}$. By the triangle inequality
$|f(t)|\le |p(\widehat\LL_k(t))|+|f(t)-p(\widehat\LL_k(t))|,$
the limit \eqref{LLp} and assumption \eqref{fperr},
we then have 
$$\lim_{t\to\infty} \frac{|f(t)|}{\iln_m(t)^{-\mu+\delta}}=0.$$
We can apply Theorem 6.2(ii) of \cite{H5} to equation \eqref{mainode}. Indeed, the same proof goes through when the nonlinear function $F(x)$ satisfies 
\eqref{Fbe} with the real number $\beta>1$ replacing $\beta=2$ in \cite{H5}.
We obtain 
\beq\label{ydec}
|y(t)|=\bigo(\iln_m(t)^{-\mu+\delta}).
\eeq
Utilizing estimate \eqref{Fbe} and choosing $\delta$ sufficiently small in \eqref{ydec} yield 
\beq\label{Gyest}|F(y(t))|=\bigo(\iln_m(t)^{\beta(-\mu+\delta)})=\bigo(\iln_m(t)^{-\mu-\delta_1}),
\eeq 
for some number $\delta_1>0$. 
We rewrite equation \eqref{mainode} as
$$y'+Ay=p(\widehat\LL_k(t))+g(t),\text{ where }g(t)=F(y(t)) + f(t)-p(\widehat\LL_k(t)).$$
With estimates \eqref{fperr} and  \eqref{Gyest}, we have $|g(t)|=\bigo(\iln_m(t)^{-\mu-\delta_2})$ with $\delta_2=\min\{\delta_0,\delta_1\}>0$. 
Then applying Lemma \ref{iterlog} to equation \eqref{mainode} results in the estimate  \eqref{abehav} from \eqref{yZp}.
\end{proof}

\subsection{The asymptotic expansions}
\label{secE}

We now turn to Theorem \ref{mainthm} and consider equation \eqref{mainode} in $\R^n$.
Assume, throughout this subsection, all Assumptions  \ref{assumpA}, \ref{assumpG}, \ref{fmain} and \ref{aspone} are true.
Hereafter, $y(t)\in\R^n$ is a   solution of \eqref{mainode} as in \eqref{soln} and \eqref{decay}.

\subsubsection*{The construction}\label{construct}
We apply Theorem \ref{thmap1} first. Condition \eqref{Fbe} is met with $\beta=\beta_1>1$ thanks to \eqref{Gyy}.
Condition \eqref{fperr} is met, thanks to the asymptotic expansions \eqref{fas1} and \eqref{fas2}, with $p=p_1$, $m=m_*$, $k=n_1$ and $\mu=\gamma_1$.
Then applying Theorem \ref{thmap1} yields
\beq \label{yq1}
|y(t)-q_1(\widehat\LL_{n_1}(t))|=\bigo(\iln_{m_*}(t)^{-\gamma_1-\delta_1})\text{ for some number $\delta_1>0$,}    
\eeq
where $q_1=\mathcal Z_A p_1$. It follows the fact $p_0 \in {\mathscr P}_{m_*}^+(n_1)$ and \eqref{strictPp} that $\alpha_{-1}=0$ in \eqref{pzdef} for the function $p_1$ and, hence, $\mathcal Z_A p_1=A^{-1}p_1$ in the definition \eqref{ZAp}. Thus,
\beq\label{qoZA}
q_1=\mathcal Z_A p_1=A^{-1}p_1.
\eeq
Combining \eqref{qoZA} with \eqref{pocond} yields 
\beq \label{q1def}
    q_1(z)=\zeta_*(z)z_{m_*}^{-\gamma_1}\xi_*,
\text{ where }
\zeta_*=p_0\in {\mathscr P}_{m_*}^+(n_1)\text{ and } \xi_*=A^{-1}\xi_0\ne 0.
\eeq
Note that $\zeta_*$ in \eqref{q1def} has the property \eqref{zestacond} in Theorem \ref{mainthm} with $n_0=n_1$.

Let $r\in \N$ and $s\in\Z_+$.  Since $F_r$ is a $C^\infty$-function in a neighborhood of $\xi_*\ne 0$, we have the following Taylor's expansion centered at $\xi_*$, for any $h\in\R^n$,
\beq\label{Taylor}
F_r(\xi_*+h)=\sum_{m=0}^s \frac1{m!}D^mF_r(\xi_*)h^{(m)}+g_{r,s}(h),
\eeq
where $D^m F_r(\xi_*) $ is the $m$-th order derivative of $F_r$ at  $\xi_*$, and
\beq\label{grs}
g_{r,s}(h)=\bigo(|h|^{s+1})\text{ as } h\to 0.
\eeq
For $m\ge 0$, denote
\beq\label{Frm}
\mathcal F_{r,m}=\frac1{m!}D^m F_r(\xi_*).
\eeq
When $m=0$, \eqref{Frm} reads as $\mathcal F_{r,0}=F_r(\xi_*)$. When $m\ge 1$, $\mathcal F_{r,m}$ is an $m$-linear mapping from $(\R^n)^m$ to $\R^n$.

By \eqref{multiL}, one has, for any $r,m\ge 1$, and $y_1,y_2,\ldots,y_m\in\R^n$, that
\beq\label{multineq}
|\mathcal F_{r,m}(y_1,y_2,\ldots,y_m)|\le \|\mathcal F_{r,m}\|\cdot |y_1|\cdot |y_2|\cdots |y_m|.
\eeq
For our convenience, we write inequality \eqref{multineq} even when $m=0$ with $\|\mathcal F_{r,0}\|\eqdef |F_r(\xi_*)|$.

Below, we rewrite the asymptotic expansions \eqref{fas1} and \eqref{fas2} in a suitable form which is crucial for the proof of Theorem \ref{mainthm2}, and hence, of Theorem \ref{mainthm}.
\begin{definition}\label{SS1}
We define a set $\widetilde {\mathcal S}\subset [0,\infty)$ as follows. 

In the case of \eqref{fas1} and \eqref{Gex}, let 
\beq \label{btil}
\widetilde\beta_j=\beta_j-1>0\text{ for }j\in\N,
\eeq 
and
\beq\label{deftilS}
\begin{aligned}
\widetilde {\mathcal S}=\Big\{& \sum_{k=1}^{\infty} m_k(\gamma_k-\gamma_1)+\sum_{j=1}^\infty \ell_j \widetilde\beta_j \gamma_1+(1-{\rm sgn}(m_*))\kappa:
m_k, \ell_j,\kappa\in \Z_+,\\
& \text{ with $m_k> 0$ for only finitely many $k$, and} \\ 
& \text{ with $\ell_j> 0$ for only finitely many $j$} \Big\}.
\end{aligned}
\eeq

In the case of \eqref{fas1} and  \eqref{Gef}, take $j=1,2,\ldots,N_*$ in \eqref{btil} and \eqref{deftilS}.

In the case of \eqref{fas2} and  \eqref{Gex}, take $k=1,2,\ldots,K$ in \eqref{deftilS}.

In the case of \eqref{fas2} and  \eqref{Gef}, take $j=1,2,\ldots,N_*$ in \eqref{btil} and \eqref{deftilS}, and take $k=1,2,\ldots,K$ in \eqref{deftilS}.

In four cases,  the set $\widetilde {\mathcal S}$ has countably, infinitely many elements. 
Arrange $\widetilde {\mathcal S}$ as a sequence $(\widetilde \mu_k)_{n=1}^\infty$ of non-negative and strictly increasing numbers.
Set 
\beqs
\mu_k=\widetilde \mu_k+\gamma_1 \text{ for $k\in\N$, and define } {\mathcal S}=\{\mu_k:k\in\N\}.
\eeqs
\end{definition}

The set $\widetilde {\mathcal S}$ has the following elementary properties.
\begin{enumerate}[label=\rnum]
\item Clearly, 
\beq \label{mu1}
    \widetilde \mu_1=0 \text{ and }\mu_1=\gamma_1.
\eeq 

\item For $\ell\in\N$ in the case of \eqref{fas1} or $1\le \ell\le K$ in the case \eqref{fas2}, by choosing $m_k=\delta_{k\ell}$, and $\ell_j=0$ for all $j$ in \eqref{deftilS}, we have $\gamma_\ell-\gamma_1\in\widetilde {\mathcal S}$. Hence, 
\beq \label{ldS}  \gamma_\ell\in {\mathcal S}\text{ for all }\ell.
\eeq

\item 
The numbers $\mu_k$ are positive and strictly increasing. Also,
\beq\label{mulim}
\widetilde\mu_k\to\infty \text{ and }\mu_k\to\infty \text{ as } k\to\infty.
\eeq

\item For all $x,y\in \widetilde {\mathcal S}$ and $k\in\N$, one has
\beq\label{Sprop} 
x+y,\ x+\widetilde\beta_k\mu_1\in\widetilde {\mathcal S}.
\eeq 
As a consequence of \eqref{Sprop}, one has  
\beq\label{muN}
\widetilde \mu_k + \widetilde\beta_j \mu_1 \ge \widetilde \mu_{k+1} \text{ for all } k,j\in\N.
\eeq

\item\label{rme} When $m_*=0$, one has $1-{\rm sgn}(m_*)=1$, hence, $\widetilde {\mathcal S}$ and ${\mathcal S}$ preserve the unit increment.
\noindent
When $m_*\ge 1$, one has $1-{\rm sgn}(m_*)=0$, hence,  ${\mathcal S}$ does not necessarily preserve the unit increment.
\end{enumerate}

For $k\in\N$, let  
\beqs
    n_k=\begin{cases} \max\{\widetilde n_j:1\le j\le k\},&\text{ in the case   \eqref{fas1},}\\
                      n_*,&\text{ in the case  \eqref{fas2}.}
     \end{cases}
\eeqs
When $k=1$, this definition of $n_1$ agrees with \eqref{n1def}.
Then one has the  the embedding 
\beq \label{PPemb}
    \widehat{\mathscr P}(\widetilde n_k,\C^n)\subset \widehat{\mathscr P}(n_k,\C^n).
\eeq
Thanks to \eqref{ldS} and by re-indexing $p_k$ and using the embedding \eqref{PPemb},
we can rewrite \eqref{fas1} and verify that $f(t)$ has the asymptotic expansion
\beq\label{mainf}
f(t) \sim  \sum_{k=1}^\infty p_k(\widehat{\LL}_{n_k}(t)),\text{ where   $p_k\in \widehat{\mathscr P}_{m_*}(n_k,-\mu_k,\C^n,\R^n)$ for  $k\in\N$.}
\eeq
Note that, because of \eqref{mu1}, $p_1$ is the same function in both \eqref{fas1} and \eqref{mainf}, or both \eqref{fas2} and \eqref{mainf}.

We now construct the functions $q_k$ in the desired asymptotic expansion \eqref{solnxp} of the solution $y(t)$.

\begin{definition}\label{qcdef}
We define $q_k(z,\zeta)$ and $\chi_k(z,\zeta)$, for $k\in \N$, recursively as follows. 

Let $q_1$ be defined by \eqref{q1def} and $\chi_1=0$.

Let $k\ge 2$ and suppose $q_j$ and $\chi_j$ are already defined for $1\le j\le k-1$. 
For $m_*\ge 1$, let $\chi_k=0$, and for $m_*=0$, let 
\beq\label{chikdef}
\chi_k(z,\zeta)=\begin{cases}
\begin{displaystyle}
\mathcal Rq_\lambda (z,\zeta) + \mathcal R\zeta_* (z)
\frac{\partial q_\lambda}{\partial\zeta}(z,\zeta) \end{displaystyle}
&\text{if there exists $\lambda\le k-1$}\\
&\text{such that $\mu_\lambda+1=\mu_k$,}\\
0&\text{otherwise.}
    \end{cases}
\eeq
Then define
\begin{multline}\label{Qkdef}
\mathcal Q_k(z,\zeta)\\
=\sum_{r\ge 1} \sum_{m=0}^\infty \sum_{\substack{k_1,\ldots,k_{m} \ge 2,\\ \sum_{j=1}^m \mu_{k_j}+(\beta_r-m)\mu_1= \mu_k}} 
(z_{m_*}^{-\mu_1}\zeta)^{\beta_r-m} 
\mathcal F_{r,m}(q_{k_1}(z,\zeta),q_{k_2}(z,\zeta),\ldots, q_{k_{m}}(z,\zeta)),
\end{multline}
and
\beq \label{qkd}    
q_k  
=\mathcal Z_A( \mathcal Q_k +p_k-\chi_k).
\eeq  
\end{definition}

The following explanations and remarks are in order.

\begin{enumerate}[label=\rnum]
\item First, we verify that  \eqref{qkd} is actually a recursive formula.
It is obvious that $\chi_k$ is defined based on the previous $q_1,\ldots,q_{k-1}$.
The last summation in \eqref{Qkdef} is over $k_1,k_2,\ldots,k_m$. The second constraint in this sum satisfies
\beq \label{muequi}
    \sum_{j=1}^m \mu_{k_j}+(\beta_r-m)\mu_1= \mu_k \text{ which is equivalent to }
\sum_{j=1}^m  \widetilde \mu_{k_j}+\widetilde \beta_r \mu_1 =\widetilde \mu_k.
\eeq
Because $\widetilde \beta_r>0$, we assert $ \widetilde \mu_{k_j}>\widetilde \mu_k$, which yields $k_j<k$, and hence
\beq \label{kjsmall}
    k_j\le k-1.
\eeq
Thus, $\mathcal Q_k$ in \eqref{Qkdef} is defined based on the previous $q_1,\ldots,q_{k-1}$. 
Consequently, so is $q_k$ in \eqref{qkd}.

\item In the case $m_*\ge 1$, one has $\chi_k=0$ for all $k\ge 1$.
In \eqref{chikdef}, the index $\lambda$, if exists, is unique. 

\item In the case of assumption \eqref{Gex}, the index $r$ in formula \eqref{Qkdef} of  $\mathcal Q_k$ is taken over the whole set $\N$.
In the case of assumption \eqref{Gef}, the index $r$ in \eqref{Qkdef} is restricted to $1,2,\ldots,N_*$.

\item When $m=0$, the terms $q_{k_j}$ in formula \eqref{Qkdef} of  $\mathcal Q_k$ are not needed, see the explanation after \eqref{Frm}, hence the condition $k_j\ge 2$ is ignored, and the last summation in  \eqref{Qkdef} becomes 
\beq\label{mzero} 
 F_r(\xi_*) \text{ for } \widetilde\beta_r \mu_1=\widetilde\mu_k,\text{ that is, } \beta_r \mu_1=\mu_k.
\eeq
Note, in \eqref{mzero}, that such an index $r$ may or may not exist. In the latter case, the term is understood to be zero. In the former case, $r$ is uniquely determined.

\item\label{checkfin} Below, we verify that the summations in \eqref{Qkdef} are over only  finitely many indices.
Let $k\ge 2$ be fixed. Firstly, thanks to \eqref{muequi}, the indices in the sum of $\mathcal Q_k$ in \eqref{Qkdef} satisfy
\beqs 
\widetilde\mu_k=\sum_{j=1}^{m} \widetilde \mu_{k_j}+\widetilde\beta_r\mu_1\ge \widetilde\beta_r \mu_1,
\eeqs
which implies
\beq\label{d1}
\widetilde\beta_r\le \widetilde\mu_k/\mu_1.
\eeq 
Secondly, for $m\ge 1$, recalling $k_j\ge 2$, one has
\beqs
\widetilde\mu_k=\sum_{j=1}^m \widetilde\mu_{k_j}+\widetilde\beta_r\mu_1>\sum_{j=1}^m \widetilde\mu_{k_j} \ge m\widetilde\mu_2,
\eeqs
which yields
\beq\label{d2}
m<\widetilde\mu_k/\widetilde\mu_2.
\eeq
Combining conditions \eqref{d1} and \eqref{d2} with \eqref{kjsmall}, one finds that the sum in \eqref{Qkdef} is over only finitely many $r$, $m$ and $k_j$. 

\item The restrictions found in \eqref{kjsmall}  and part \ref{checkfin} above also result in the following consequence.
For $k\ge 2$, suppose $r^*,s^*,k^*$ are non-negative integers such that 
\beq\label{sumcond} 
\widetilde\beta_{r^*}\ge \widetilde \mu_k/\mu_1,\  s^*\ge \widetilde\mu_k/\widetilde\mu_2,\ k^*\ge k-1.
\eeq 
Then $\mathcal Q_k$ can be equivalently written as 
\beq\label{finitesum}
\mathcal Q_k(z,\zeta)=\sum_{r=1}^{r^*} \sum_{m=0}^{s^*} \sum_{\substack{2\le k_1,k_2,\ldots,k_m \le k^*, \\ \sum_{j=1}^m \widetilde\mu_{k_j}+\widetilde\beta_r\mu_1=\widetilde\mu_k}}
(z_{m_*}^{-\mu_1}\zeta)^{\beta_r-m} 
\mathcal F_{r,m}(q_{k_1}(z,\zeta),q_{k_2}(z,\zeta),\ldots, q_{k_{m}}(z,\zeta)). 
\eeq

\item In the case \eqref{Gef} and $k\ge 2$,  formula \eqref{finitesum} can be recast as
\beq\label{fsf}
\mathcal Q_k(z,\zeta)=\sum_{r=1}^{N_*} \sum_{m=0}^{s^*} \sum_{\substack{2\le k_1,k_2,\ldots,k_m \le k^*, \\ \sum_{j=1}^m \widetilde\mu_{k_j}+\widetilde\beta_r\mu_1=\widetilde\mu_k}}
(z_{m_*}^{-\mu_1}\zeta)^{\beta_r-m} 
\mathcal F_{r,m}(q_{k_1}(z,\zeta),q_{k_2}(z,\zeta),\ldots, q_{k_{m}}(z,\zeta)), 
\eeq
for any non-negative integers $s^*,k^*$ satisfying  
\beq\label{scf} 
s^*\ge \widetilde\mu_k/\widetilde\mu_2\text{ and } k^*\ge k-1.
\eeq 

\item Define, for $k\ge 2$,
\beq\label{qtil}
\widetilde q_k(z,\zeta)= \zeta^{-1} z_{m_*}^{\mu_k}q_k (z,\zeta), \text{i.e., }
q_k(z,\zeta)= \zeta z_{m_*}^{-\mu_k}\widetilde q_k (z,\zeta).
\eeq
Then we can rewrite \eqref{finitesum} as
\begin{align*}
&\mathcal Q_k(z,\zeta)\\
&=\sum_{r=1}^{r_*} \sum_{m=0}^{s_*} \sum_{\substack{2\le k_1,\ldots,k_{m} \le k^*,\\ \sum_{j=1}^m \mu_{k_j}+(\beta_r-m)\mu_1=\mu_k}} 
z_{m_*}^{(m-\beta_r)\mu_1-\sum_{j=1}^m \mu_{k_j}} \zeta^{\beta_r}  \mathcal F_{r,m}(\widetilde q_{k_1}(z,\zeta),\widetilde q_{k_2}(z,\zeta),\ldots, \widetilde q_{k_{m}}(z,\zeta)),  
\end{align*}
which gives
\beq \label{Qqtil}
{\mathcal Q}_k(z,\zeta)=z_{m_*}^{-\mu_k} \sum_{r=1}^{r_*} \sum_{m=0}^{s_*} \sum_{\substack{2\le k_1,\ldots,k_{m} \le k^*,\\ \sum_{j=1}^m \widetilde\mu_{k_j}+\widetilde \beta_r\mu_1= \widetilde \mu_k}} 
\zeta^{\beta_r}  \mathcal F_{r,m}(\widetilde q_{k_1}(z,\zeta),\widetilde q_{k_2}(z,\zeta),\ldots, \widetilde q_{k_{m}}(z,\zeta)).
\eeq
\end{enumerate}

\subsubsection*{Proof of the main result}
There is still a question whether $q_k$ and $\chi_k$ belong to classes of functions in Definitions \ref{Fclass} or \ref{Fsubo} so that we can apply operators $\mathcal R$ in \eqref{chikdef} and $\mathcal Z_A$ in \eqref{qkd}. 
We answer this question in the next proposition.

\begin{proposition}\label{qcpropo}
For any $k\in \N$, 
\beq \label{qcclaim}
q_k\in \widehat{\mathscr P}_{m_*}(n_k,-\mu_k,\C^n,\R^n)\text{  and }  
\chi_k\in \widehat{\mathscr P}_{m_*}(n_k,-\mu_k,\C^n,\R^n).    
\eeq
\end{proposition}
\begin{proof}
We prove \eqref{qcclaim} by induction.
It is clear that \eqref{qcclaim} holds true for $k=1$.
Let $k\ge 2$. Suppose 
\beq \label{qchypo}
q_j\in \widehat{\mathscr P}_{m_*}(n_j,-\mu_j,\C^n,\R^n)\text{  and }  
\chi_j\in \widehat{\mathscr P}_{m_*}(n_j,-\mu_j,\C^n,\R^n)\text{ for }1\le j\le k-1.    
\eeq

Consider the case when  $m_*=0$ and $\chi_k$ is given by the first formula in \eqref{chikdef}. Thanks to the hypothesis \eqref{qchypo} and Lemma \ref{invar2}, we have 
\begin{align*}
&\mathcal Rq_\lambda \in \widehat{\mathscr P}_{m_*}(n_k,-\mu_\lambda-1,\C^n,\R^n)=\widehat{\mathscr P}_{m_*}(n_k,-\mu_k,\C^n,\R^n) ,\\
&\mathcal R\zeta_* \in \widehat{\mathscr P}_{m_*}(n_k,-1,\C,\R),\quad 
\frac{\partial q_\lambda}{\partial\zeta}\in  \widehat{\mathscr P}_{m_*}(n_k,-\mu_\lambda,\C^n,\R^n). 
\end{align*}
Using the last two properties and applying Corollary \ref{cplxcor}, we obtain that the product 
$$\mathcal R\zeta_* \frac{\partial q_\lambda}{\partial\zeta}\text{ belongs to }
\widehat{\mathscr P}_{m_*}(n_k,-\mu_\lambda-1,\C^n,\R^n)=\widehat{\mathscr P}_{m_*}(n_k,-\mu_k,\C^n,\R^n).$$
Thus, 
\beq \label{chigood}
    \chi_k\in \widehat{\mathscr P}_{m_*}(n_k,-\mu_k,\C^n,\R^n).
\eeq
When $\chi_k=0$ in either case $m_*\ge 1$ or $m_*=0$, \eqref{chigood} trivially holds.

Recalling the complexification in Definition \ref{cmplxify} and the mappings $\mathcal F_{r,m}$ in \eqref{Frm}, we define 
\beqs
    \mathcal F_{r,m,\C}=
    \begin{cases}
        \mathcal F_{r,0}=\mathcal F_r(\xi_*), &\text{ for } m=0,\\
        \text{the complexification of $\mathcal F_{r,m}$,}& \text{ for }m\ge 1.
    \end{cases}
\eeqs

We examine the construction of $q_k$ in \eqref{Qkdef} and \eqref{qkd}.
Because the functions $q_{k_j}$ are $\R^n$-valued, we rewrite \eqref{Qkdef} as
\begin{multline}\label{QkC}
\mathcal Q_k(z,\zeta)\\
=\sum_{r\ge 1} \sum_{m=0}^\infty \sum_{\substack{k_1,\ldots,k_{m} \ge 2,\\ \sum_{j=1}^m \mu_{k_j}+(\beta_r-m)\mu_1= \mu_k}} 
(z_{m_*}^{-\mu_1}\zeta)^{\beta_r-m} 
\mathcal F_{r,m,\C}(q_{k_1}(z,\zeta),q_{k_2}(z,\zeta),\ldots, q_{k_{m}}(z,\zeta)),
\end{multline}

When $m=0$, the summand in \eqref{QkC} becomes, thanks to \eqref{mzero}, 
$$z_{m_*}^{-\mu_1\beta_r}\zeta^{\beta_r} 
\mathcal F_{r}(\xi_*)=z_{m_*}^{-\mu_k}\zeta^{\beta_r} 
\mathcal F_{r}(\xi_*)$$
which is a function in $\widehat{\mathscr P}_{m_*}(n_k,-\mu_k,\C^n,\R^n)$.

Consider $m\ge 1$.
By the property \eqref{MCbar} applied to $\mathcal F_{r,m,\C}$, hypothesis \eqref{qchypo} and applying  Lemma \ref{prodclass}, we have
\beqs
    \mathcal F_{r,m,\C}(q_{k_1}(z,\zeta),q_{k_2}(z,\zeta),\ldots, q_{k_{m}}(z,\zeta))
    \in \widehat{\mathscr P}_{m_*}(n_k,-\sum_{j=1}^m \mu_{k_j},\C^n,\R^n).
\eeqs
Thus, the summand $(z_{m_*}^{-\mu_1}\zeta)^{\beta_r-m} 
\mathcal F_{r,m,\C}(\cdots)$ in \eqref{QkC} belongs to $\widehat{\mathscr P}_{m_*}(n_k,-\mu,\C^n,\R^n)$, where
$$\mu=\sum_{j=1}^m \mu_{k_j}+\mu_1(\beta_r-m)=\mu_k.$$ 

In both cases $m=0$ and $m\ge 1$, by summing up the summands  in \eqref{QkC}  finitely many times, see \eqref{finitesum}, we obtain 
\beq \label{Qgood}
    \mathcal Q_k\in \widehat{\mathscr P}_{m_*}(n_k,-\mu_k,\C^n,\R^n).
\eeq
Thanks to the property of $p_k$ in \eqref{mainf}, property \eqref{chigood} of $\chi_k$ and \eqref{Qgood} of $\mathcal Q_k$, we have 
\beqs
    \mathcal Q_k+p_k-\chi_k\in \widehat{\mathscr P}_{m_*}(n_k,-\mu_k,\C^n,\R^n).
\eeqs
Then by formula \eqref{qkd} and Lemma \ref{invar2}(ii), we obtain  
$q_k\in \widehat{\mathscr P}_{m_*}(n_k,-\mu_k,\C^n,\R^n)$.

By the Induction Principle, we have \eqref{qcclaim} holds true for all $k\in\N$.
\end{proof}

Let $\psi(t)=\iln_{m_*}(t)$. Recall that $\Xi_k(t)$ is defined in \eqref{Zbar}, which, in fact, is the same as \eqref{Xikbar}, for $k\ge n_0$.
As a consequence of Proposition \ref{qcpropo} and definition \eqref{qtil}, we have $\widetilde q_k\in \widehat{\mathscr P}_{m_*}(n_k,0,\C^n,\R^n)$, which implies, thanks to \eqref{phates},
\beq \label{qtilest}
    |\widetilde q_k\circ \Xi_{n_k}(t)|=\bigo(\psi(t)^\varep)\text{ for any }\varep>0.
\eeq

Clearly, $p_k$  and $q_1$ are also functions of $z$ and $\zeta$ by defining 
$$p_k(z,\zeta)=p_k(z)\text{ and }q_1(z,\zeta)=q_1(z).$$ 
For $k\ge 1$, let 
$$
f_k(t)=p_k(\widehat{\LL}_{n_k}(t))=p_k(\Xi_{n_k}(t)),\ 
y_k(t)=q_k(\Xi_{n_k}(t))
\text{ and }
v_k(t)=y(t)-\sum_{j=1}^k y_j(t).$$
Define 
\beqs 
\theta(t)=(\zeta z_{m_*}^{-\mu_1})\circ \Xi_{n_1}(t)=\zeta_*(\widehat{\LL}_{n_1}(t))\psi(t)^{-\mu_1}.
\eeqs 
With this notation, we have from \eqref{q1def} that 
$$q_1(\Xi _{n_1}(t))=\zeta _*(\widehat{\mathcal {L}}_{n_1}(t))L_{m_*}(t)^{-\mu _1} \xi_*=\theta(t)\xi_*.$$
For $k\in\N$, denote
\beqs
\widetilde y_k(t)=\theta(t)^{-1}  y_k(t)  \text { and } 
\widetilde v_k(t)=\theta(t)^{-1}  v_k(t).
\eeqs
 Recalling that $\widetilde q_k$ is defined by \eqref{qtil}, one has  
\beqs
\widetilde y_k(t)= \theta(t)^{-1}q_k(\Xi_{n_k}(t)) 
=\zeta_*(\widehat{\LL}_{n_1}(t))^{-1} \psi(t)^{\mu_1}\cdot \zeta_*(\widehat{\LL}_{n_1}(t)) \psi(t)^{-\mu_k}  \widetilde  q_k(\widehat{\LL}_{n_1}(t)) ,
\eeqs
 which gives
\beq\label{Ovk1}
\widetilde y_k(t)= \psi(t)^{-\widetilde \mu_k}  \widetilde q_k(\Xi_{n_k}(t)) .
\eeq
Combining \eqref{Ovk1} with \eqref{qtilest} yields
 \beq\label{Ovk2}
\widetilde y_k(t) =\bigo(\psi(t)^{-\widetilde\mu_k+\varepsilon})
=\bigo(\psi(t)^{\varepsilon})\text{ for any number $\varepsilon >0$.}
\eeq

For any $N\in\N$, we write
\begin{align*}
    y(t)&=q_1(\Xi_{n_1}(t))+\sum_{k=2}^N y_k(t) +v_N(t)
    = \theta(t)\xi_* + \sum_{k=2}^N \theta(t)\widetilde y_k(t) +\theta(t)\widetilde v_N(t),
\end{align*}
which gives
\beq \label{ysplit}
     y(t)= \theta(t)\left \{ \xi_* + \sum_{k=2}^N \widetilde y_k(t) +\widetilde v_N(t)\right\}.
\eeq

The following theorem is a detailed version of the Main Theorem \ref{mainthm}.

\begin{theorem}\label{mainthm2}
With $f(t)$ having the asymptotic expansion \eqref{mainf} and $q_k$ defined in Definition \ref{qcdef}, the statement \eqref{qkhc} and asymptotic expansion   \eqref{solnxp} hold true.
\end{theorem}
\begin{proof}
Property \eqref{qkhc} was already proved in Proposition \ref{qcpropo}. We prove the asymptotic expansion \eqref{solnxp} now. Below, $t$ is sufficiently large.

\medskip\noindent
Part I. Consider $m_*=0$.
We will prove for the case \eqref{Gex} first, and then for the case \eqref{Gef}.

\medskip\textit{Case A: Assume \eqref{Gex}.}
For any $N\in \N$, we denote by $(\mathcal T_N)$ the following statement
 \beq\label{inhypo}
\left|y(t) - \sum_{k=1}^N q_k(\Xi_{n_k}(t)) \right|=\bigo(\psi(t)^{-\mu_N-\delta_N})\quad \text{as }t\to\infty,
\eeq
for some $\delta_N>0$. 

We will prove $(\mathcal T_N)$  for all $N\in\N$ by induction in $N$.  

\medskip
\textbf{First step  ($N=1$).} 
Thanks to \eqref{yq1} and the fact $\mu_1=\gamma_1$, the statement $(\mathcal T_1)$ is true.

\medskip
Applying estimate \eqref{plusrate} to $p=\zeta_*$, we have 
\beq \label{zetsdel}
    \zeta_*(\widehat{\LL}_{n_1}(t))^{s} =\bigo(\psi(t)^{\delta})\text{ for any }s\in\R,\delta>0.
\eeq 

Also, estimate \eqref{ydec} becomes
\beq\label{ydec2}
|y(t)|=\bigo(\psi(t)^{-\mu_1+\delta}) \text{ for all }\delta>0.
\eeq

\medskip\textbf{Induction step.} Let $N\ge 1$.  Suppose the statement $(\mathcal T_N)$ holds true. 
By $(\mathcal T_1)$ and the Induction Hypothesis $(\mathcal T_N)$,  one has
\beq\label{uNbo}
   v_1(t)= \bigo(\psi(t)^{-\mu_1 - \delta_1}), \quad  v_N(t)= \bigo(\psi(t)^{-\mu_N - \delta_N}).
\eeq

We derive the differential equation for $v_N(t)$. Compute
\begin{align*}
v_N'
&=y'-\sum_{k=1}^N y_k'
=-Ay + F(y)+f -\sum_{k=1}^N y_k'\\ 
&=-Av_N -\sum_{k=1}^N A y_k + F(y) +f -\sum_{k=1}^N y_k'  .
\end{align*}
Together with the fact
\beqs
f(t)=\sum_{k=1}^{N+1}f_k(t)+\bigo(\psi(t)^{-\mu_{N+1}-\varep_0})\text{ for some number }\varep_0>0,
\eeqs 
we obtain
\beq\label{eqN}
v_N'+Av_N= F(y) - \sum_{k=1}^N (A y_k+y_k')+\sum_{k=1}^{N+1}f_k(t)+\bigo(\psi(t)^{-\mu_{N+1}-\varep_0}).
\eeq

By \eqref{mulim}, we can choose a number $r_*\in\N$ such that 
\beq \label{pchoice}
\beta_{r_*}\ge \mu_{N+1}/\mu_1, \text{ which is equivalent to }
\widetilde\beta_{r_*}\ge \widetilde\mu_{N+1}/\mu_1.
\eeq 
By \eqref{Ner},  one has 
\beq\label{Fcut}
F(x)=\sum_{r=1}^{r_*} F_r (x)+\bigo(|x|^{\beta_{r_*}+\varep_* }) \text{ as }x\to 0, \text{ for some }\varep_*>0.
\eeq
Using \eqref{Fcut} with $x=y(t)$ and utilizing estimate \eqref{ydec2}, we write 
\beq\label{Fy53}
F(y(t))
=\sum_{r=1}^{r_*} F_r (y(t))+\bigo(|y(t)|^{\beta_{r_*}+\varep_* })
=\sum_{r=1}^{r_*} F_r (y(t))+\bigo(\psi(t)^{(-\mu_1+\delta)(\beta_{r_*}+\varep_*)}),
\eeq
for any number $\delta>0$.
Choose $\delta>0$ such that
$$\delta\le \frac{\mu_1\varep_*}{2(\beta_{r_*}+\varep_*)}.$$
Together with condition \eqref{pchoice}, we can estimate the last power in \eqref{Fy53} by 
$$
(\mu_1-\delta)(\beta_{r_*}+\varep_*)=\mu_1\beta_{r_*}+(\mu_1\varep_*-\delta(\beta_{r_*}+\varep_*))
\ge \mu_{N+1} + \mu_1\varep_*/2.
$$
We obtain
\beq\label{Fy}
F(y(t))
=\sum_{r=1}^{r_*} F_r (y(t))+\bigo(\psi(t)^{-\mu_{N+1}-\mu_1\varep_*/2}).
\eeq

We calculate the sum $\sum_{r=1}^{r_*} F_r(y)$.
By estimate \eqref{plusrate} applied to $p=\zeta_*$, $s=-1$, $\delta=\delta_N/2$, and \eqref{uNbo}, we have
 \beq\label{Ountilde1}
\widetilde v_N(t)= \bigo( \zeta_*(\widehat{\LL}_{n_1}(t))^{-1} \psi(t)^{\mu_1}\cdot \psi(t)^{-\mu_N-\delta_N})
= \bigo(\psi(t)^{-\widetilde\mu_N - \delta_N/2} ),
\eeq
and similarly
\beq\label{util1}
\widetilde v_1(t)= \bigo(\psi(t)^{-\widetilde\mu_1 - \delta_1/2} ) =\bigo(\psi(t)^{-\delta_1/2 } ).
\eeq
Applying formula \eqref{ysplit} to $N=1$ yields 
$y(t)=\theta(t)(\xi_*+\widetilde v_1),$
hence
\beq\label{Frv}
F_r(y(t))
= F_r\big( \theta(t) (\xi_* + \widetilde v_1 ) \big)
=  \theta(t)^{\beta_r} F_r(\xi_* + \widetilde v_1).
\eeq

Let $s_*\in\N$ satisfy 
\beq \label{schoice} 
s_*\delta_1/2 + \beta_1\mu_1\geq \mu_{N+1}
\text{ and } s_*\ge \widetilde\mu_{N+1}/\widetilde\mu_2. 
\eeq 
By Taylor's expansion \eqref{Taylor} with $s=s_*$, using the notation in \eqref{Frm},   
\beq\label{Frxi}
F_r (\xi_* + \widetilde v_1)
 = \sum_{m=0}^{s_*} \mathcal F_{r,m}\widetilde v_1^{(m)}+ g_{r,s_*}(\widetilde v_1)
 =F_r(\xi_*)+\sum_{m=1}^{s_*} \mathcal F_{r,m}\widetilde v_1^{(m)}+ g_{r,s_*}(\widetilde v_1).
\eeq
It follows \eqref{Frv} and \eqref{Frxi} that
\beq\label{Fvex1}
F_r(y(t))= \theta(t)^{\beta_r}\sum_{m=0}^{s_*} \mathcal F_{r,m}\widetilde v_1(t)^{(m)}
+\theta(t)^{\beta_r} g_{r,s_*}(\widetilde v_1(t)).
\eeq 

For the last term in \eqref{Fvex1}, by using \eqref{grs} and the decay of $|\widetilde v_1(t)|$ from \eqref{util1},  we find that
\beqs
\theta(t)^{\beta_r}  g_{r,s_*}(\widetilde v_1(t))
=\zeta_*(\widehat{\LL}_{n_1}(t))^{\beta_r}\psi(t)^{-\beta_r\mu_1}  \bigo(|\widetilde v_1(t)|^{s_*+1}).
\eeqs
Utilizing \eqref{zetsdel} gives 
$$\zeta_*(\widehat{\LL}_{n_1}(t))^{\beta_r} =\bigo(\psi(t)^{\delta_1/4}).$$
Combining this estimate with \eqref{util1} and the first condition in \eqref{schoice}, we derive 
\begin{align}\label{eg}
\theta(t)^{\beta_r}  g_{r,s_*}(\widetilde v_1(t))
&=  \bigo( \psi(t)^{\delta_1/4}\psi(t)^{-\beta_r\mu_1}\psi(t)^{-\delta_1(s_*+1)/2}) \notag\\
&=  \bigo(\psi(t)^{-(\beta_1\mu_1+\delta_1s_*/2+\delta_1/4)})=\bigo ( \psi(t)^{-\mu_{N+1}-\delta_1/4}).
\end{align}

Next, we compute $\theta(t)^{\beta_r}\sum_{m=0}^{s_*} \mathcal F_{r,m}\widetilde v_1(t)^{(m)}$ in \eqref{Fvex1}.
Consider $m\ge 1$.
Since 
$$v_1=\sum_{k=2}^N y_k +v_N,\text{ we have }
 \widetilde v_1=\sum_{k=2}^N \widetilde y_k +\widetilde v_N.$$
 Hence, we can calculate
\begin{align}
\mathcal F_{r,m}\widetilde v_1^{(m)}
&=\mathcal F_{r,m}\Big(\sum_{k=2}^N \widetilde y_k +\widetilde v_N\Big)^{(m)}
=\mathcal F_{r,m}\Big(\sum_{k=2}^N \widetilde y_k +\widetilde v_N,
\sum_{k=2}^N \widetilde y_k +\widetilde v_N,\ldots,
\sum_{k=2}^N \widetilde y_k +\widetilde v_N\Big) \notag\\
&=\mathcal F_{r,m}\Big(\sum_{k=2}^N \widetilde y_k\Big)^{(m)}+\sum_{\rm finitely\ many}\mathcal F_{r,m}(h_1,\ldots,h_m). \label{u1m}
\end{align}
Note, in the case $N=1$, that the sum $\sum_{k=2}^N \widetilde y_k$ and, hence, the term  $\mathcal F_{r,m}(\sum_{k=2}^N \widetilde y_k)^{(m)} $ are not present in the calculations in \eqref{u1m}.

In the last sum of \eqref{u1m}, each $h_1,\ldots,h_m$ is either $\sum_{k=2}^N \widetilde y_k$ or $\widetilde v_N$, and at least one of these $h_j$ must be $\widetilde v_N$.
Suppose there are $\ell$ terms $\widetilde v_N$ among $h_1,\ldots,h_m$ for some integer $\ell\in[1,m]$. 
By inequality \eqref{multineq} and the estimate \eqref{Ovk2} for $\widetilde y_k$,  we have
\beqs
|\mathcal F_{r,m}(h_1(t),\ldots,h_m(t))|\le \|\mathcal F_{r,m}\| \cdot |h_1(t)| \ldots |h_m(t)| 
= \bigo(\psi(t)^{\varep(m-\ell)}|\widetilde v_N(t)|^\ell) 
\eeqs
for any $\varep>0$. Using estimate \eqref{Ountilde1} for $\widetilde v_N(t)$ and choosing $\varep$ sufficiently small if needed, we obtain
\beqs
|\mathcal F_{r,m}(h_1(t),\ldots,h_m(t))|
= \bigo(\psi(t)^{\varep(m-1)}|\widetilde v_N(t)|) 
= \bigo(\psi(t)^{\delta_N/4}\psi(t)^{-\widetilde\mu_N - \delta_N/2})= \bigo(\psi(t)^{-\widetilde\mu_N-\delta_N/4} ). 
\eeqs
Combining this estimate with \eqref{u1m}  and formula \eqref{Ovk1}, we obtain 
\begin{align*}&\sum_{m=0}^{s_*} \mathcal F_{r,m}\widetilde v_1(t)^{(m)}
=F_r(\xi_*) + \sum_{m=1}^{s_*} \mathcal F_{r,m}\Big(\sum_{k=2}^N \widetilde y_k(t)\Big)^{(m)} + \bigo(\psi(t)^{-\widetilde\mu_N-\delta_N/4} )\\
&=F_r(\xi_*) + \sum_{m=1}^{s_*} \sum_{k_1,\ldots,k_m = 2}^N\mathcal F_{r,m}(\widetilde y_{k_1}(t),\widetilde y_{k_2}(t),\ldots, \widetilde y_{k_m}(t))
+ \bigo(\psi(t)^{-\widetilde\mu_N-\delta_N/4} )\\
&=F_r(\xi_*) + \sum_{m=1}^{s_*} \sum_{k_1,\ldots,k_m = 2}^N \psi(t)^{-\sum_{j=1}^m \widetilde\mu_{k_j}}\mathcal F_{r,m}(\widetilde q_{k_1},\widetilde q_{k_2},\ldots, \widetilde q_{k_m}) \circ \Xi_{n_N}(t)
+  \bigo(\psi(t)^{-\widetilde\mu_N-\delta_N/4} ). 
\end{align*}
Thus,
\beq\label{esumF}
\begin{aligned}
& \theta(t)^{\beta_r}  \sum_{m=0}^{s_*} \mathcal F_{r,m}\widetilde v_1(t)^{(m)}\\
&=\sum_{m=0}^{s_*} \sum_{k_1,\ldots,k_m=2}^N \psi(t)^{-(\sum_{j=1}^m \widetilde\mu_{k_j}+\beta_r\mu_1)}  \zeta_*(\widehat{\LL}_{n_1}(t))^{\beta_r}  \mathcal F_{r,m}(\widetilde q_{k_1},\widetilde q_{k_2},\ldots,\widetilde  q_{k_m}) \circ \Xi_{n_N}(t)\\
&\quad + R(t), 
\end{aligned}
\eeq 
where
$R(t)=\bigo( \zeta_*(\widehat{\LL}_{n_1}(t))^{\beta_r}\psi(t)^{-\widetilde\mu_N-\beta_r\mu_1-\delta_N/4)} )$.
Again, in the case $N=1$, the last double summation in \eqref{esumF} has only one term corresponding to $m=0$, which is $F_r(\xi_*)$.

Applying estimate \eqref{zetsdel} to $s=\beta_r$ and $\delta=\delta_N/8$, we have 
$\zeta_*(\widehat{\LL}_{n_1}(t))^{\beta_r} =\bigo(\psi(t)^{\delta_N/8})$.
Then
$|R(t)|=\bigo( \psi(t)^{-(\widetilde\mu_N+\beta_r\mu_1+\delta_N/8)} )$.
Using property \eqref{muN}, we have 
\beqs 
\widetilde\mu_N + \beta_r\mu_1+ \delta_N/8=
\widetilde\mu_N + \widetilde\beta_r\mu_1+\mu_1+ \delta_N\ge \widetilde \mu_{N+1}+\mu_1+\delta_N/8=\mu_{N+1}+\delta_N/8. 
\eeqs 
Hence, the last term in \eqref{esumF} can be estimated as 
\beq\label{ee}
|R(t)|= \bigo(\psi(t)^{-\mu_{N+1} -\delta_N/8} ).
\eeq 

Therefore, by  \eqref{Fvex1}, \eqref{eg}, \eqref{esumF}, \eqref{ee}, we have 
\beq\label{Tmore}
\sum_{r=1}^{r_*} F_r (y(t))
=J(t)+\bigo(\psi(t)^{-\mu_{N+1} -\delta_N/8} )+\bigo(\psi(t)^{-\mu_{N+1}-\delta_1/4} )
=J(t)+\bigo(\psi(t)^{-\mu_{N+1}-\varep_1} ),
\eeq
where $\varep_1=\min\{\delta_N/8,\delta_1/4\}>0$, and 
\beq\label{Jdef}
J(t)=\sum_{r=1}^{r_*}\sum_{m=0}^{s_*} \sum_{k_1,\ldots,k_m =2}^N \psi(t)^{-(\sum_{j=1}^m \widetilde\mu_{k_j}+\beta_r\mu_1)}  \zeta_*(\widehat{\LL}_{n_1}(t))^{\beta_r} \mathcal F_{r,m}(\widetilde q_{k_1},\widetilde q_{k_2},\ldots, \widetilde q_{k_m})\circ \Xi_{n_N}(t). 
\eeq

Regarding the power of $\psi(t)^{-1}$ in \eqref{Jdef}, 
denote $\widetilde\mu= \widetilde\mu_{k_1}+\ldots +\widetilde\mu_{k_m}+\widetilde\beta_r\mu_1 $. 

When $m=0$, one has $\widetilde\mu=\widetilde\beta_r\mu_1$, which belongs to  $\widetilde{\mathcal S}$.
When $m\ge 1$, by property \eqref{Sprop}, $\widetilde\mu$ also belongs to $\widetilde{\mathcal S}$.
Clearly, $\widetilde\mu>0=\widetilde \mu_1$.
Thus, in both cases of $m$,  there is a unique integer $p\ge 2$ such that
\beq \label{mumu}
    \widetilde\mu=\widetilde\mu_p.
\eeq
Because of the indices $r,m,k_1,\ldots,k_m$ being finitely many, there are only finitely many such numbers $p$. Thus, there is  $p_*\in\N$ such that any index $p$ in \eqref{mumu} must satisfy $p\le p_*$.
Hence, the exponent of $\psi(t)^{-1}$ in \eqref{Jdef} is
\beq \label{mul}
\widetilde\mu+\mu_1=\widetilde\mu_p+\mu_1=\mu_p\in\mathcal S  \quad \text{ for some integer }  p\in[2,p_*].
\eeq

Using the index $p$ in \eqref{mul}, we can split the sum in $J(t)$ into two parts corresponding to $p \le N+1$ and $p\ge N+2$.
We then write accordingly  $$J(t)=S_1(t)+S_2(t),$$
where 
\begin{align*}
S_1(t)&=\sum_{p=2}^{N+1}\sum_{r=1}^{r_*}\sum_{m=0}^{s_*}  \sum_{\substack{2\le k_1,\ldots,k_m\le N,\\  \sum_{j=1}^m \widetilde\mu_{k_j}+\beta_r\mu_1=\mu_p}}
\psi(t)^{-\mu_p } \zeta_*(\widehat{\LL}_{n_1}(t))^{\beta_r}  \mathcal F_{r,m}(\widetilde q_{k_1},\widetilde q_{k_2},\ldots, \widetilde q_{k_m})\circ \Xi_{n_N}(t),\\
S_2(t)&=\sum_{p=N+2}^{p_*}\sum_{r=1}^{r_*}\sum_{m=0}^{s_*} \sum_{\substack{2\le k_1,\ldots,k_m\le N, \\ \sum_{j=1}^m \widetilde\mu_{k_j}+\beta_r\mu_1=\mu_p}} \psi(t)^{-\mu_p } \zeta_*(\widehat{\LL}_{n_1}(t))^{\beta_r}  \mathcal F_{r,m}(\widetilde q_{k_1},\widetilde q_{k_2},\ldots, \widetilde q_{k_m})\circ \Xi_{n_N}(t).
\end{align*}

Defining 
\beq\label{Jk}
J_k(t)=\psi(t)^{-\mu_k }  \sum_{r=1}^{r_*} \sum_{m=0}^{s_*} \sum_{\substack{2\le k_1,\ldots,k_{m} \le N,\\ \sum_{j=1}^m \widetilde\mu_{k_j}+\beta_r\mu_1= \mu_k}}  \zeta_*(\widehat{\LL}_{n_1}(t))^{\beta_r}  \mathcal F_{r,m}(\widetilde q_{k_1},\widetilde q_{k_2},\ldots, \widetilde q_{k_{m}})\circ \Xi_{n_N}(t) 
\eeq
for $k=2,\dots,N+1$, we re-write 
\beq \label{S1f}
    S_1(t)=\sum_{k=2}^{N+1}J_k(t).
\eeq

We estimate $S_2(t)$ next. Set $\varep_2=\min\{\varep_0,\varep_1,\mu_1\varep_*/2,(\mu_{N+2}-\mu_{N+1})/2\}>0$.
Utilizing inequality \eqref{zetsdel} to estimate 
$$\zeta_*(\widehat{\LL}_{n_1}(t))^{\beta_r} =\bigo(\psi(t)^{\varep_2/2}).$$
Using inequality \eqref{multineq} to estimate $|\mathcal F_{r,m}(\widetilde q_{k_1},\widetilde q_{k_2},\ldots, \widetilde q_{k_m})\circ \Xi_{n_N}(t)|$, and applying inequality \eqref{qtilest} to $\widetilde q_{k_j}\circ \Xi_{n_N}(t)$, we have 
\begin{align*}
|\mathcal F_{r,m}(\widetilde q_{k_1},\widetilde q_{k_2},\ldots, \widetilde q_{k_m})\circ \Xi_{n_N}(t)|
&\le \|\mathcal F_{r,m}\| \cdot|\widetilde q_{k_1}\circ \Xi_{n_N}(t)|\cdot |\widetilde q_{k_2}\circ \Xi_{n_N}(t)|\ldots |\widetilde q_{k_m}\circ \Xi_{n_N}(t)|  \\
&=\bigo(\psi(t)^{\varep_2/2}).
\end{align*}
For the term $\psi(t)^{-\mu_p}$ in the formula of $S_2(t)$, we use the fact $\mu_p\ge \mu_{N+2}$. Combining the above estimates, we obtain 
\beq\label{S2f}
S_2(t)=\bigo( \psi(t)^{-\mu_{N+2}}\psi(t)^{\varep_2/2}\psi(t)^{\varep_2/2})=\bigo( \psi(t)^{-\mu_{N+2}+\varep_2})
=\bigo( \psi(t)^{-\mu_{N+1}-\varep_2}).
\eeq 

It follows \eqref{S1f} and \eqref{S2f} that
\beq\label{Jest}
J(t)=\sum_{k=2}^{N+1}J_k(t)+\bigo( \psi(t)^{-\mu_{N+1}-\varep_2}).
\eeq
Combining \eqref{Tmore} with \eqref{Jest} gives
\beq\label{T1e}
\sum_{r=1}^{r_*} F_r (y(t))=\sum_{k=2}^{N+1}J_k(t)+\bigo( \psi(t)^{-\mu_{N+1}-\varep_2 }).
\eeq

Thus, by \eqref{eqN}, \eqref{Fy} and \eqref{T1e}, we obtain
\beq\label{vpr}
v_N'+A v_N
=\sum_{k=2}^{N+1}J_k(t)-\sum_{k=1}^N (A y_k+y_k') +\sum_{k=1}^{N+1}f_k(t)
+\bigo( \psi(t)^{-\mu_{N+1}-\varep_2}).
\eeq

By formula \eqref{dqze0},  it holds, for $k\in\N$, that
\beq \label{ykeq2}    
y_k'=\ddt q_k(\Xi_{n_k}(t))
=(\mathcal M_{-1}q_k+\mathcal R q_k+ \mathcal R\zeta_*\frac{\partial q_k}{\partial\zeta})\circ \Xi_{n_k}(t).
\eeq  
Summing up \eqref{ykeq2} in $k$ gives
\beq\label{sumypr}
\sum_{k=1}^N y_k'=\sum_{k=1}^N \mathcal M_{-1}q_k\circ \Xi_{n_k}(t) + \sum_{\lambda=1}^N (\mathcal R q_\lambda+ \mathcal R\zeta_*\frac{\partial q_\lambda}{\partial\zeta})\circ \Xi_{n_\lambda}(t).
\eeq
Note that we already made a change of the index notation from $k$ to $\lambda$ for the last sum.

Regarding  the last sum in \eqref{sumypr},
we observe that $\mathcal Rq_\lambda\in \mathscr F_{0}(n_\lambda,-\mu_\lambda-1,\C^n)$. Thanks to remark \ref{rme} after Definition \ref{SS1}, we have $\mu_\lambda+1\in\mathcal S$. Hence, there exists a unique number $k\in\N$ such that $\mu_k=\mu_\lambda+1$. Because $\mu_k>\mu_\lambda$, we have $\lambda\le k-1$. Thus, 
$$\mathcal R q_\lambda+ \mathcal R\zeta_*\frac{\partial q_\lambda}{\partial\zeta}=\chi_k.$$
Considering three possibilities $k\le N$, $k=N+1$ and $k\ge N+2$, we  rewrite, similar to \eqref{Jest},
\beq\label{sumR}
\sum_{\lambda=1}^N (\mathcal R q_\lambda+ \mathcal R\zeta_*\frac{\partial q_\lambda}{\partial\zeta})\circ \Xi_{n_\lambda}(t)
=\sum_{k=1}^N \chi_k\circ \Xi_{n_k}(t) +\chi_{N+1}\circ \Xi_{n_{N+1}}(t)+\bigo(\psi(t)^{-\mu_{N+1}-\varep_3})
\eeq
for some number $\varep_3\in(0,\varep_2]$. Therefore,
\beq\label{shorty}
\sum_{k=1}^N y_k'=\sum_{k=1}^N (\mathcal M_{-1}q_k+\chi_k)\circ \Xi_{n_k}(t) +\chi_{N+1}\circ \Xi_{n_{N+1}}(t) +\bigo(\psi(t)^{-\mu_{N+1}-\varep_3}).
\eeq

Combining \eqref{vpr} and \eqref{shorty} yields 
\beq\label{vN3}
v_N'+Av_N
= f_{N+1}(t) - \sum_{k=1}^N X_k(t) - \chi_{N+1}\circ \Xi_{n_{N+1}}(t)+  J_{N+1}(t) +\bigo(\psi(t)^{-\mu_{N+1}-\varep_3}),
\eeq 
 where
\beqs
X_1(t)= (Aq_1+\mathcal M_{-1}q_1 + \chi_1-p_1)\circ \Xi_{n_1}(t),
\eeqs
\beqs
X_k(t)= (Aq_k+\mathcal M_{-1}q_k + \chi_k-p_k)\circ \Xi_{n_k}(t)- J_k(t) \text{ for } 2\le k\le N.
\eeqs

We already know $\chi_1=0$ and $(A+\mathcal M_{-1})q_1=(A+\mathcal M_{-1})\mathcal Z_Ap_1=p_1$. Hence, 
\beq \label{X1z}
    X_1=0.
\eeq

Consider $k\ge 2$.  
Note that the condition $\sum_{j=1}^m \widetilde\mu_{k_j}+\beta_r\mu_1= \mu_k$ in formula \eqref{Jk} of $J_k(t)$ is exactly \eqref{muequi}.
Then, for each $k=2,\ldots,N+1$, we utilize the identity \eqref{Qqtil} for $k\le N+1$, $r^*=r_*$, $s^*=s_*$ and $k^*=N$, noticing that condition \eqref{sumcond} is met thanks to the condition for $\widetilde\beta_{r_*}$ in \eqref{pchoice}, the second condition for $s_*$ in \eqref{schoice}, and the fact $k^*=N\ge k-1$.
It results in 
\beqs
\mathcal Q_k(\Xi_{n_k}(t))=J_k(t) \text{ for $k=2,\ldots,N+1$ .}
\eeqs

By the identity \eqref{ZAM}, we can write
\beqs
\chi_k=(A+\mathcal M_{-1})\mathcal Z_A\chi_k,\
p_k=(A+\mathcal M_{-1})\mathcal Z_A p_k,\
J_k=((A+\mathcal M_{-1})\mathcal Z_A\mathcal Q_k)\circ \Xi_{n_k}.
\eeqs
Therefore,
\beqs
X_k(t)=\Big [(A+\mathcal M_{-1})( q_k + \mathcal Z_A(\chi_k-p_k-\mathcal Q_k))\Big] \circ \Xi_{n_k}(t).
\eeqs 
For $2\le k\le N$, one has from \eqref{qkd} that
$q_k + \mathcal Z_A(\chi_k-p_k-\mathcal Q_k)=0$,
hence, 
\beq \label{Xkz}
    X_k=0.
\eeq

Combining equation \eqref{vN3} with \eqref{X1z} and \eqref{Xkz} gives
\beq\label{vN5}
v_N'+Av_N 
=(p_{N+1}- \chi_{N+1}+ \mathcal Q_{N+1})\circ \Xi_{n_{N+1}}(t) 
 +\bigo(\psi(t)^{-\mu_{N+1}-\varep_3} ).
\eeq
Applying Theorem \ref{iterlog} to equation \eqref{vN5} yields
\beqs
|v_N(t)-(\mathcal Z_A ( \mathcal Q_{N+1} + p_{N+1} -\chi_{N+1}))\circ \Xi_{n_{N+1}}(t)|=\bigo(\psi(t)^{-\mu_{N+1}-\delta_{N+1}})
\eeqs
for some number $\delta_{N+1}>0$.
Note that 
\beqs
(\mathcal Z_A ( \mathcal Q_{N+1} + p_{N+1} -\chi_{N+1}))\circ \Xi_{n_{N+1}}=q_{N+1}\circ \Xi_{n_{N+1}}=y_{N+1}.
\eeqs
Therefore,
\beqs
|v_N(t)-y_{N+1}(t)|=\bigo(\psi(t)^{-\mu_{N+1}-\delta_{N+1}}),
\eeqs
which implies 
\beqs
\left|y(t)-\sum_{k=1}^{N+1} y_k(t)\right|=\bigo(\psi(t)^{-\mu_{N+1}-\delta_{N+1}}).
\eeqs
Thus, \eqref{inhypo} is true for $N:=N+1$.
Hence the statement  $(\mathcal T_{N+1})$ holds true.

\medskip\textbf{Conclusion for Case A.} 
By the Induction Principle, the statement  \eqref{inhypo}  is true for all $N\in\N$. 
Therefore, we obtain the asymptotic expansion \eqref{solnxp}.

\medskip
\textit{Case  B: Assume \eqref{Gef}.} 
We follow the  proof in Case A with the following adjustments.
The number $r_*$ is simply $N_*$, and condition \eqref{pchoice} for $r_*$  is not required anymore. All the sum $\sum_{r\ge 1}$ appearing in the proof that involves $F_r$ or $\mathcal F_{r,m}$  will be replaced with $\sum_{1\le r\le N_*}$. From \eqref{Fcut} to the end of the proof in Case A, positive number $\varep_*$ is arbitrary, and number $\beta_{r_*}$ in calculations from \eqref{Fcut} to the end is replaced with any number $\beta_*\ge \mu_{N+1}/\mu_1$. Then \eqref{Fcut} still holds true thanks to \eqref{errF2}.
We also take into account that $\mathcal Q_k$ now is given by \eqref{fsf} under condition \eqref{scf}.
With these changes, the above proof for Case A goes through, and we obtain the desired statement for this case B.

This ends Part I of the proof for the case $m_*=0$.

\medskip    
\noindent Part II. Consider $m_*\ge 1$. We follow the proof of Part I. 
By the virtue of identity \eqref{dqze2},  it holds, for  $k\in\N$, that
\beq \label{ykeq3}    
y_k'(t)=\mathcal M_{-1} q_k(\Xi_{n_k}(t))+\bigo(t^{-\gamma}) \quad\forall\gamma\in(0,1).
\eeq  
In \eqref{sumypr}, by using \eqref{ykeq3} instead of \eqref{ykeq2}, we replace
$$\sum_{\lambda=1}^N (\mathcal R q_\lambda+ \mathcal R\zeta_*\frac{\partial q_\lambda}{\partial\zeta})\circ \Xi_{n_\lambda}(t)
\text{ with }\bigo(t^{-\gamma}).$$ 
Because a function of $\bigo(t^{-\gamma})$  is also of $\bigo(\psi(t)^{-\mu_{N+1}-\delta})$ for any $\delta>0$, then after neglecting \eqref{sumR}, all terms $\chi_k$ for $1\le k\le N+1$ in calculations thereafter can be taken to be $0$.
Thus, the proof can be proceeded as in Part I. 
 The proof of Theorem \ref{mainthm2} is now complete.
\end{proof}

\begin{remark}The following are remarks on Theorem \ref{mainthm2}.
\begin{enumerate}[label=\rnum]
    \item     Since $q_1(\widehat{\LL}_{n_1}(t))=z_{m_*}^{-\mu_1}\zeta \xi_*|_{(z,\zeta)=\Xi_{n_1}(t)}$, we can alternatively take in Definition \ref{qcdef} 
    $$q_1=q_1(z,\zeta)=z_m^{-\mu_1}\zeta \xi_*.$$

    \item If the function $\zeta_*$ and all $p_k$ in \eqref{fas1} are restricted to having only real power vectors $\alpha$ for the variable $z$, then the functions $q_k$ in the asymptotic  expansion  \eqref{solnxp} that are constructed by Definition  \ref{qcdef} also have only real power vectors $\alpha$ for the variable $z$,.

    \item Suppose $p_0(z)=z^\alpha$ for a real power vector $\alpha$.
    Then we only need to find $q_k=q_k(z) \in \mathscr P_m(n_k,-\mu_k,\C^n,\R^n)$.
    Indeed, with the substitution $\zeta=z^\alpha$, formula \eqref{qtil} gives  $\tilde q_k=\tilde q_k(z)$ and formula \eqref{Qqtil} gives $\mathcal Q_k=\mathcal Q_k(z)$, both functions belong to $\mathscr P_m(n_k,-\mu_k,\C^n,\R^n)$. Hence, there is no need to use the variable $\zeta$. In fact, in this case, the asymptotic expansion \eqref{solnxp} can be achieved with $q_k=q_k(z)$ by solely using the techniques previously developed  in \cite{CaHK1,H5}.

    \item In the proof of Theorem \ref{mainthm2}, we only need the Taylor series for each function $F_k$ 
    around $\xi_*=A^{-1}\xi_0\ne 0$. Therefore, the condition $F_k\in C^\infty(\R^n\setminus\{0\})$ in Assumption \ref{assumpG}\ref{GG} can be replaced with $F_k$ being a $C^\infty$-function in a neighborhood of $\xi_*$.
\end{enumerate}
\end{remark}

\begin{remark}
Our methodology can also be applied to the cases that either
\begin{itemize}
    \item the function $F(x)$ has a finite sum approximation, as $x\to 0$,  in the sense of \eqref{errF2} with \emph{some} number $\beta>\beta_{N_*}$, or 
    \item the forcing function $f(t)$ has a finite sum asymptotic approximation, as $t\to\infty$, in the sense of \eqref{gremain2} with \emph{some} number $\mu>\gamma_N$.  
\end{itemize}
In these situations, one can follow the steps in this section, especially the proof of Theorem \ref{mainthm2}, 
to find a finite sum asymptotic approximation   for the solutions $y(t)$. (See some details in  \cite[Theorem 4.1]{CaH1} and \cite[Theorem 5.1]{CaHK1}.)
\end{remark}

\begin{example}\label{examples}
Let the matrix $A$ be as in Assumption \ref{assumpA}. Suppose
$$F(x)=|x|^{3/2}Z_0 + 4|x|^{3/4}x,\text{ for some constant vector }
Z_0\in \R^n\setminus\{0\}.$$

\begin{enumerate}

\item Assume
$$f(t)=\left ( (\ln t)^{1/2} - (\ln\ln t)^{-3} +1\right )t^{-1} \xi_0+ (\ln t)^{1/3}t^{-3/2}\xi_1\text{ with } 
\xi_0,\xi_1\in \R^n\setminus\{0\}.$$
Let $m_*=0$, $\widetilde n_k=n_k=2$ and  $\zeta_*(z_{-1},z_0,z_1,z_2)=p_0(z_{-1},z_0,z_1,z_2)=z_1^{1/2}-z_2^{-3}+1$.
Set 
$\zeta(t)=\zeta_*(\widehat \LL_2(t))= (\ln t)^{1/2} - (\ln\ln t)^{-3} +1$. Then any decaying solution $y(t)$ of \eqref{mainode} has the following  asymptotic expansion, as $t\to\infty$, written in the form \eqref{clearpower} 
\beqs
y(t)\sim \sum_{k=1}^\infty q_k( \widehat \LL_2(t),\zeta(t)) t^{-\mu_k},\text{ where } q_k\in \widehat{\mathscr P}_0(2,0,\C^n,\R^n)
\eeqs

\item Assume
$$f(t)=\left ( -3(\ln\ln t)^{1/2}L_4(t)^{-3} + 2L_3(t)^{-3/2}\right)(\ln t)^{-2} \xi_0+ L_4(t)^{3/5}(\ln t)^{-5}\xi_1 $$
 with $\xi_0,\xi_1\in \R^n\setminus\{0\}$. Let $m_*=1$, $\widetilde n_k=n_k=4$ and, for $z=(z_{-1},z_0,z_1,z_2,z_3,z_4)$, 
$$\zeta_*(z)=p_0(z)=3z_2^{1/2}z_4^{-3}-2z_3^{-3/2}.$$
Set $\zeta(t)=\zeta_*(\widehat \LL_4(t)) = 3(\ln\ln t)^{1/2}L_4(t)^{-3} - 2L_3(t)^{-3/2}$. Then any decaying solution $y(t)$ of \eqref{mainode} has  the following asymptotic expansion as $t\to\infty$, written in the form \eqref{expan5}
\beqs
y(t)\sim \sum_{k=1}^\infty q_k( \widehat \LL_4(t),\zeta(t))(\ln t)^{-\mu_k},\text{ where } q_k\in \widehat{\mathscr P}_1(4,0,\C^n,\R^n).
\eeqs
\end{enumerate}

\end{example}

\section{Statements with explicit sinusoidal functions}\label{sinsec}

The purpose of this section is to rephrase Theorem \ref{mainthm} using only real numbers. This seemingly simple task, however, involves some technicalities that require rigorous treatments.
First, we characterize the classes ${\mathscr P}_{m}(k,0,\C^n,\R^n)$ and $\widehat{\mathscr P}_{m}(k,0,\C^n,\R^n)$  more directly using real number sinusoidal functions instead of complex number power functions.

\begin{definition}\label{sinusclass}
Let $X$ be a real linear space.
Given integers $k\ge m\ge 0$.
\begin{enumerate}[label=\tnum]
    \item Define the class $\widehat{\mathcal P}_m^1(k,X)$ to be the collection of functions $p(z,\zeta)$ which are the finite sums of the following functions
\beq\label{realpz}
(z,\zeta)\in (0,\infty)^{k+2}\times(0,\infty)\mapsto z^\alpha \zeta^\beta \prod_{j=0}^k \sigma_j(\omega_j z_j)\xi,
\eeq
where $\xi\in X$, $z=(z_{-1},z_0,\ldots,z_k)$,
 $\alpha=(\alpha_{-1},\alpha_0,\ldots,\alpha_k)\in \mathcal E_\R(m,k,0)$, $\beta\in\R$,
$\omega_j\in\R$,  
and, for each $j$,  the function  $\sigma_j$ is either cosine or sine.

\item Define the class $\widehat{\mathcal P}_m^0(k,X)$ to be the subset of $\widehat{\mathcal P}_m^1(k,X)$ when all numbers $\omega_j$ in \eqref{realpz} are zero. 

\item Define $\mathcal P_m^1(k,X)$, respectively, $\mathcal P_m^0(k,X)$, to be the set of functions $p=p(z)$  in $\widehat{\mathcal P}_m^1(k,X)$, respectively, $\widehat{\mathcal P}_m^0(k,X)$, that is, $\zeta$ and $\beta$ are omitted in \eqref{realpz}. 
\end{enumerate}
\end{definition}

Note in \eqref{realpz} that the components of the power vector $\alpha$ satisfy 
$$\alpha_{-1}=\alpha_0=\ldots=\alpha_m=0.$$
Consequently, $z_{-1}^{\alpha_{-1}}=1$ in \eqref{realpz}.

The next class of functions is a counter part of that in Definition \ref{Pplus}.

\begin{definition}\label{sinplus}
Given integers $k\ge m\ge -1$.
Let $\mathcal P_m^{+}(k)$ be the set of functions $p\in \mathcal P_m^1(k,\R)$ such that
$$p(z)=p_{\max}(z)+q(z),$$
where $p_{\max}\in \mathcal P_m^0(k,\R)$ is of the form 
$$p_{\max}(z)=z^{\alpha_{\max}} c_{\max} \text{ with } \alpha_{\max}\in \mathcal E_\R(m,k,0)\text{ and }  c_{\max}>0,$$
and $q\in \mathcal P_m^1(k,\R)$  is a finite sum of functions
\beq \label{qsmall}
   z\in(0,\infty)^{k+2}\mapsto z^\alpha  \prod_{j=0}^k \sigma_j(\omega_j z_j)\xi \text{ similar to \eqref{realpz}, } 
\eeq
with $\alpha<\alpha_{\max}$ and $\omega_0=0$.
\end{definition}

We observe that the facts in Remarks (a)--(c) after Definition \ref{Pplus} still hold true when 
$ {\mathscr P}^+(k)$ and $ {\mathscr P}_m^+(k)$  are replaced with ${\mathcal P}_m^+(k)$.
Also, thanks to the condition $\alpha\in \mathcal E_\R(m,k,0)$ for \eqref{realpz}, if $p\in {\mathcal P}_m^1(k,\R^n)$, $q\in \widehat{\mathcal P}_m^1(k,\R^n)$ and $\zeta_\bullet\in {\mathcal P}_m^+(k)$, then one has, similar to estimate \eqref{phates}, 
\beq\label{samees}   
 \left|p(\widehat\LL_{k}(t))\right|=\bigo(\iln_{m}(t)^{s})\text{ and } 
   \left|  q(\widehat\LL_{k}(t),\zeta_\bullet(\widehat\LL_{k}(t)))\right|
   =\bigo(\iln_{m}(t)^{s}) \text{ for any }s>0.
\eeq

With these observations, we can define the asymptotic expansions involving functions in Definitions \ref{sinusclass} and \ref{sinplus} as the following.

\begin{definition}\label{realxp}
Let $(X,\|\cdot\|_X)$ be a real normed space.
Let $m_*\in \Z_+$ and $g(t)$ be a function from  $(T,\infty)$ to $X$,  for some $T\in\R$. 
\begin{enumerate}[label=\tnum]
\item  The asymptotic expansions 
\beqs
g(t)\sim \sum_{k=1}^\infty  g_k(\widehat\LL_{n_k}(t))\iln_{m_*}(t)^{-\gamma_k}, 
\text{ where $ g_k\in \mathcal P_{m_*}^1(n_k,\R^n) $ for $k\in\N$, }
\eeqs
and  
\beqs
g(t)\sim \sum_{k=1}^N g_k(\widehat{\LL}_{n_*}(t))\iln_{m_*}(t)^{-\gamma_k}, \text{ where $g_k\in \mathcal P_{m_*}^1(n_*,X)$ for $1\le k\le N$, }
\eeqs
are defined in the same ways as  Definition \ref{Lexpand}, part \ref{LE1} and \ref{LE2}, repectively.

\item Let $\zeta_\bullet\in \mathcal P_{m_*}^{+}(n_0)$ for some $n_0\ge m_*$, and denote 
\beq \label{newZbar}
    \Xi_k(t)=\left (\widehat{\LL}_{k}(t),\zeta_\bullet(\widehat{\LL}_{n_0}(t))\right ).
\eeq
Then the asymptotic expansions
\beqs
g(t)\sim \sum_{k=1}^\infty  g_k(\Xi_{n_k}(t))\iln_{m_*}(t)^{-\gamma_k}, 
\text{ where $ g_k\in \widehat{\mathcal P}_{m_*}^1(n_k,\R^n) $ for $k\in\N$, }
\eeqs
and
\beqs
g(t)\sim \sum_{k=1}^N g_k(\Xi_{n_*}(t))\iln_{m_*}(t)^{-\gamma_k}, \text{ where $g_k\in \widehat{\mathcal P}_{m_*}(n_*,X)$ for $1\le k\le N$, }
\eeqs 
are defined in the same ways as Definition \ref{zetxp}, part \ref{ze1} and  \ref{ze2}, respectively.
\end{enumerate}
\end{definition}

The meaning of the asymptotic expansions in  Definition \ref{realxp} can be justified by the using the estimates in \eqref{samees}.
We can now  restate Theorem \ref{mainthm} using these forms of  asymptotic expansions.

\begin{theorem}\label{thmreal}
Let Assumptions \ref{assumpA} and \ref{assumpG} hold true.
Let $m_*\in\Z_+$ be given and assume the following additionally.
\begin{enumerate}[label=\tnum]
    \item The function  $f\in C([T,\infty),\R^n)$, for some $T\ge 0$, has the asymptotic expansion, in the sense of Definition \ref{realxp}{\rm (i)},
\beq\label{freal1}
f(t)\sim \sum_{k=1}^\infty g_k(\widehat \LL_{\widetilde n_k}(t))\iln_{m_*}(t)^{-\gamma_k}, 
\text{ where } g_k\in \mathcal P_{m_*}^1(\widetilde n_k,\R^n) \text{ for }k\in\N,
\eeq
or
\beq\label{freal2}
f(t)\sim \sum_{k=1}^K g_k(\widehat{\LL}_{n_*}(t))\iln_{m_*}(t)^{-\gamma_k}, 
\text{ where } g_k\in \mathcal P_{m_*}^1(n_*,\R^n) \text{ for }1\le k\le K.
\eeq

\item There exist a function $g_0\in \mathcal P_{m_*}^+(n_1)$ and a vector $\xi_0\in\R^n\setminus\{0\}$ such that
\beqs 
    g_1(z)=g_0(z)\xi_0,
\eeqs
where $n_1$ is $\widetilde n_1$ in the case \eqref{freal1}, and  is $n_*$ in the case \eqref{freal2}.
\end{enumerate}

Then there exist a divergent, strictly increasing sequence $(\mu_k)_{k=1}^\infty\subset (0,\infty)$,
an increasing sequence $(n_k)_{k=0}^\infty\subset \Z_+\cap[m_*,\infty)$, a function
$\zeta_\bullet\in \mathcal P_{m_*}^+(n_0)$,
and functions
 $ h_k\in \widehat{\mathcal P}_{m_*}^1(n_k,\R^n)$  for all $k\in\N$,
  such that any solution  $y(t)\in\R^n$ of \eqref{mainode} as in \eqref{soln}  and \eqref{decay} admits the asymptotic expansion  
 \beq\label{yreal2}
y(t)\sim \sum_{k=1}^\infty h_k\left(\widehat{\LL}_{n_k}(t),\zeta_\bullet(\widehat{\LL}_{n_0}(t))\right)\iln_{m_*}(t)^{-\mu_k}
 \text{ in the sense of Definition \ref{realxp}{\rm (ii)}.}
\eeq    
\end{theorem}

The proof of  Theorem \ref{thmreal} will use the following conversions between the functions in 
subsections \ref{oldex}, \ref{newex} and this section \ref{sinsec}.

\begin{lemma}\label{convert}
Given integers $k\ge m\ge 0$. 
\begin{enumerate}[label=\tnum]

\item\label{cv2} If $p\in {\mathcal P}_{m}^1(k,\R^n) $ then
\beq\label{PPR2}
p\circ \widehat{\LL}_k = q\circ \widehat{\LL}_k \text{ for some } q\in {\mathscr P}_{m}(k,0,\C^n,\R^n).
\eeq
In fact, the function $q(z)$ is of the form
\beq \label{spec1}
    q(z)=\sum_{\alpha\in S} z^\alpha\xi_\alpha \text{ as in Definition \ref{realF}, }\alpha=(\alpha_{-1},\alpha_0,\ldots,\alpha_{k}),\Im(\alpha_{k})=0.
\eeq

    \item\label{cv1} Suppose $p\in {\mathscr P}_{m}(k,0,\C^n,\R^n) $. Then
\beq\label{PPR1}
p\circ  \widehat{\LL}_k=q\circ  \widehat{\LL}_{k+1}\text{ for some } q\in {\mathcal P}_{m}^1(k+1,\R^n).
\eeq
Moreover, if $p(z)$ is of the form \eqref{spec1},
then
$q\in {\mathcal P}_{m}^1(k,\R^n)$.

\item \label{cv5}
If $p\in\mathcal P_m^{+}(k)$, then 
\beq \label{samezz}
p(\widehat{\LL}_k(t))=q(\widehat{\LL}_k(t)) \text{ for some function } 
q\in \mathscr P_m^{+}(k).
\eeq

\item\label{cv4} If $p\in \widehat{\mathcal P}_{m}^1(k,\R^n) $ and $\zeta_\bullet\in\mathcal P_m^{+}(k)$, then
\beqs
p\left (\widehat{\LL}_k(t),\zeta_\bullet(\widehat{\LL}_k(t))\right) 
= q\left (\widehat{\LL}_k(t),\zeta_\bullet(\widehat{\LL}_k(t))\right) \text{ for some } q\in \widehat{\mathscr P}_{m}(k,0,\C^n,\R^n).
\eeqs
More specifically, $q(z,\zeta)$ is of the form
\beq\label{spec2}
q(z,\zeta)=\sum_{(\alpha,\beta)\in S} z^\alpha\zeta^\beta \xi_{\alpha,\beta} \text{ as in Definition \ref{realF}, }\alpha=(\alpha_{-1},\alpha_0,\ldots,\alpha_{k}),\Im(\alpha_{k})=0. 
\eeq

\item\label{cv3} Suppose $p\in \widehat{\mathscr P}_{m}(k,0,\C^n,\R^n) $ and $\zeta_\bullet\in\mathcal P_m^{+}(k)$. Then
\beqs
p\left (\widehat{\LL}_k(t),\zeta_\bullet(\widehat{\LL}_k(t))\right) 
= q\left (\widehat{\LL}_{k+1}(t),\zeta_\bullet(\widehat{\LL}_k(t))\right)
\text{ for some } q\in \widehat{\mathcal P}_{m}^1(k+1,\R^n).
\eeqs
Moreover, if $p(z,\zeta)$ is of the form \eqref{spec2},
then
$q\in \widehat{\mathcal P}_{m}^1(k,\R^n)$.
\end{enumerate}
\end{lemma}
\begin{proof}
The statements (i) and (ii) were proved in \cite[Section 10]{H5} and \cite[Section 7]{H6} with brief arguments. We present here a detailed account.

\medskip
 Part  \ref{cv2}. For $z$ and $\alpha$ as in \eqref{realpz}, the function
    \beq \label{realgood}
        z\mapsto z^\alpha\text{ belongs to }\mathscr P_{m}(k,0,\C,\R).
    \eeq
    Observe that
 \beq\label{sincos1}
  \cos(\omega\iln_j(t))=g(\widehat \LL_k(t))\text{ and }\sin(\omega\iln_j(t))=h(\widehat \LL_k(t)),
   \eeq
 where
 \beq\label{sincos2} 
 g(z)=\frac12(z_{j-1}^{i\omega}+z_{j-1}^{-i\omega})\text{ and }
h(z)=\frac1{2i}(z_{j-1}^{i\omega}-z_{j-1}^{-i\omega}).
\eeq
Obviously, the powers $i\omega$ and $(-i\omega)$ are imaginary numbers, and 
\beq\label{sincos3} 
g,h\in \mathscr P_{m}(k,0,\C,\R).
\eeq
Using properties \eqref{realgood}, \eqref{sincos1}, \eqref{sincos3} and Corollary \ref{cplxcor}, we find that each function 
\beqs
\tilde p(z)\eqdef z^\alpha\prod_{j=0}^k \sigma_j(\omega_j z_j)\xi 
\eeqs
satisfies
$\tilde p\circ \widehat{\LL}_k(t)=\tilde q\circ \widehat{\LL}_k(t)$,  where $\tilde q\in\mathscr P_{m}(k,0,\C^n,\R^n)$.
Summing up finitely many times such functions $\tilde p$, we obtain \eqref{PPR2}.

The second statement is due to the fact that  the functions  $\cos(\omega\iln_k(t))$ and $\sin(\omega\iln_k(t))$ can be converted via  \eqref{sincos1} and \eqref{sincos2}, when $j=k$,  using the functions of   the variable $z_{k-1}$.

\medskip
Part \ref{cv1}.  First, we observe, for $-1\le j\le k$,  $\omega\in\R$  and $\xi\in\C^n$, that
\beq\label{LXY}
\iln_j(t)^{i\omega}\xi =X(t)+iY(t),\quad 
\iln_j(t)^{i\omega}\xi+\iln_j(t)^{-i\omega}\bar \xi=2X(t),
\eeq
where 
\begin{align*}
    X(t)&=\cos(\omega \iln_{j+1}(t))\Re\xi-\sin(\omega \iln_{j+1}(t))\Im\xi,\\
    Y(t)&=\sin(\omega \iln_{j+1}(t))\Re\xi+\cos(\omega \iln_{j+1}(t))\Im\xi.
\end{align*}
Now, thanks to \eqref{phalf}, it suffices to prove \eqref{PPR1} for 
 $$p(z)=z^\alpha \xi+z^{\bar \alpha }\bar \xi,\quad \alpha=r+i\omega,\quad 
 r\in\mathcal E_\R(m,k,0),\quad \omega\in \R^{k+2},\quad \xi\in\C^n.$$
Denote $\alpha=(\alpha_{-1},\alpha_{0},\ldots,\alpha_k)$
and $\omega=(\omega_{-1},\omega_{0},\ldots,\omega_k).$
We write explicitly
$$\widehat{\LL}_k(t)^{i\omega}\xi=\prod_{\ell=-1}^k \iln_\ell(t)^{i\omega_\ell}\xi.$$
To compute this, we define, for $-1\le j\le k$,  $Z_j(t)=\prod_{\ell=-1}^j \iln_\ell(t)^{i\omega_\ell}\xi $. 
For convenience, we also denote $Z_{-2}(t)=\xi$. 
We then have  $\widehat{\LL}_k(t)^{i\omega}\xi=Z_k(t)$ and
\beq \label{ZZ}
    Z_{j}(t)=\iln_{j}(t)^{i\omega_{j}}Z_{j-1}(t)\text{ for $-1\le j\le k$.}
\eeq
Define, for $-2\le j\le k$, 
$$V_j(t)=\begin{pmatrix}V_j^{(1)}(t)\\ V_j^{(2)}(t)\end{pmatrix}
\eqdef \begin{pmatrix}\Re Z_j(t)\\\Im Z_j(t)\end{pmatrix},\text{ and  } 
D(t)=\begin{pmatrix}
    \cos(t)I_n& -\sin(t) I_n\\ \sin(t) I_n&\cos(t) I_n
\end{pmatrix}.$$
Then 
\beq \label{pLV}
    p\circ \widehat{\LL}_k(t) =\widehat{\LL}_k(t)^r \left( Z_k(t)+\overline{Z_k(t)}\right)
    =2\widehat{\LL}_k(t)^r V_k^{(1)}(t).
\eeq
Using \eqref{LXY}, we rewrite \eqref{ZZ} as 
$V_j(t)=D(\omega_{j}\iln_{j+1}(t)) V_{j-1}(t).$
Recursively, we obtain 
\beq\label{Vkform}
\begin{aligned} 
V_{k}(t)&=D(\omega_{k}\iln_{k+1}(t))V_{k-1}(t)=D(\omega_{k}\iln_{k+1}(t))D(\omega_{k-1}\iln_{k}(t))V_{k-2}(t)=\ldots\\ 
&= D(\omega_{k}\iln_{k+1}(t))D(\omega_{k-1}\iln_{k}(t))\ldots D(\omega_{-1} \iln_0(t))V_{-2}(t)\\
&= D(\omega_{k}\iln_{k+1}(t))D(\omega_{k-1}\iln_{k}(t))\ldots D(\omega_{-1} t)(\Re\xi,\Im\xi).
\end{aligned}
\eeq 
Combining \eqref{pLV} with \eqref{Vkform}, we obtain \eqref{PPR1}.

For the second statement of \ref{cv1},  if $p(z)$ is of the form \eqref{spec1}, then $\omega_k=0$, 
which implies $D(\omega_{k}\iln_{k+1}(t))=I_{2n}$ in \eqref{Vkform}. Thus, we have $q\in {\mathcal P}_{m}^1(k,\R^n)$.

\medskip
Part \ref{cv5}. This statement is a consequence of part \ref{cv2} and the fact that $\omega_0=0$ in \eqref{qsmall} for $p$ implies $\alpha_{-1}=0$ in \eqref{strictPp} for $q$.

\medskip
The proofs of parts \ref{cv4} and \ref{cv3} are the same as those of parts \ref{cv2} and \ref{cv1}, respectively. The reason is that  the extra factor $\zeta^\beta$ does not affect the arguments which are based on the variable $z$. 
\end{proof}

\begin{proof}[Proof of Theorem \ref{thmreal}]
By the virtue of Lemma \ref{convert}\ref{cv2} and \ref{cv5}, the asymptotic expansion \eqref{freal1} or \eqref{freal2} of $f(t)$ can be converted to the asymptotic expansion \eqref{fas1} or \eqref{fas2}, respectively, and Assumptions \ref{fmain} and \ref{aspone} hold true. In particular, thanks to \eqref{samezz}, we have
$$g_0(\widehat{\LL}_{n_0}(t))=p_0(\widehat{\LL}_{n_0}(t))\text{ for some function $p_0\in \mathscr P_{m_*}^{+}(n_0)$ with $n_0=n_1$.}$$
With $\zeta_*=p_0$ in \eqref{q1def}, we take $\zeta_\bullet=g_0$ to have
$\zeta_\bullet(\widehat{\LL}_{n_0}(t))=\zeta_*(\widehat{\LL}_{n_0}(t)).$
Hence, $\Xi_k(t)$ defined by \eqref{Zbar} agrees with  \eqref{newZbar}.

Applying Theorem \ref{mainthm}, we obtain the asymptotic expansion \eqref{solnxp}, which, by Lemma \ref{convert}\ref{cv3}, can be converted to the asymptotic expansion \eqref{yreal2}, where the number $n_k$ in \eqref{yreal2} is adjusted by adding $1$ to $n_k$ in \eqref{solnxp}.
\end{proof}

\begin{remark}\label{noextra}
    Suppose $f(t)$ has the asymptotic expansion \eqref{freal1}, where $\widetilde n_k=n_*$ for a fixed number $n_*$ and all $k\ge 1$. Then in \eqref{mainf} $n_k=n_*$ and, thanks to the second statement of Lemma \ref{convert}\ref{cv2}, all function $p_k(z)$ in \eqref{fas1} is of the form \eqref{spec1}.
    Going through the construction in subsection \ref{secE}, $n_k=n_*$  in \eqref{solnxp}, and the function $\zeta_*$ in \eqref{q1def} and  all functions $q_k$  in \eqref{solnxp} are of the form \eqref{spec2}. By Lemma \ref{convert}\ref{cv3}, all $n_k=n_*$ in \eqref{yreal2}, that is, we \emph{do not} need to use $n_k=n_*+1$ in the last sentence of the proof of Theorem \ref{thmreal}.  
    The case  $f(t)$ has the asymptotic expansion \eqref{freal2} is similar.
\end{remark}

\section{Conclusions}\label{conclude}
We have obtained the asymptotic expansion, in a new form,  for decaying solutions of non-autonomous nonlinear non-smooth differential equations. We specified a suitable  class of complicated decaying forcing functions for this theory to be applicable. The novelty is the introduction of a new subordinate variable in establishing such a closed-form expansion. The use of the subordinate variable in the current paper may enable us to obtain asymptotic expansions in other situations that are deemed impossible in previous study.  This idea will be explored in our future work. It can be applied to both nonlinear ODE and PDE. In fact, we develop the theory for multi-subordinate variables even for the NSE in a subsequent paper \cite{H11}.  Our result may be extended to equations with fractional derivatives, even for newer classes such as those in \cite{frac1,frac2}.
    
\medskip
\noindent\textbf{Acknowledgment.} The author would like to thank Dat Cao for many insightful  discussions.

\medskip
\noindent\textbf{Data availability.} No new data were created or analyzed in this study.

\medskip
\noindent\textbf{Funding.} No funds were received for conducting this study,

\medskip
\noindent\textbf{Conflict of interest.} There are no conflicts of interests.

\bibliography{paperbaseall}{}
\bibliographystyle{abbrv}

 \end{document}